\documentclass[11pt,reqno]{amsart}
\usepackage{bm}
\usepackage{amssymb}
\usepackage{graphicx}
\usepackage{multirow}
\usepackage{enumitem}
\usepackage{mathtools}
\usepackage[margin=3cm]{geometry}
\usepackage{xcolor,colortbl}
\usepackage{makecell}
\usepackage{hyperref}
\hypersetup{
	colorlinks=true,         
	linkcolor=blue,          
	citecolor=blue,         
	filecolor=blue,          
	urlcolor=blue           
}
\usepackage[tableposition=below]{caption} 
\newcommand{\R}{\mathbb{R}}
\newcommand{\Z}{\mathbb{Z}}
\newcommand{\V}{\mathbf{V}}

\newcommand{\calR}{\mathcal{R}}
\newcommand{\calJ}{\mathcal{J}}
\newcommand{\calC}{\mathcal{C}}
\newcommand{\fm}{\mathfrak{m}}
\newcommand{\tS}{\mathtt{S}}
\newcommand{\tR}{\mathtt{R}}
\newcommand{\tJ}{\mathtt{J}}
\newcommand{\tI}{\mathtt{I}}
\newcommand{\calS}{\mathcal{S}}
\newcommand{\br}{\bm{r}}
\newcommand{\bs}{\bm{s}}
\DeclareMathOperator{\im}{im}
\DeclareMathOperator{\lk}{Lk}
\renewcommand\geq\geqslant
\renewcommand\leq\leqslant 
\newcommand{\St}{\mathrm{star}\,}
\theoremstyle{plain}
\newtheorem{theorem}{Theorem}[section]
\newtheorem{corollary}[theorem]{Corollary}
\newtheorem{lemma}[theorem]{Lemma}
\newtheorem{proposition}[theorem]{Proposition}
\theoremstyle{remark}
\newtheorem{remark}[theorem]{Remark}
\newtheorem{example}[theorem]{Example}
\theoremstyle{definition}
\newtheorem{definition}[theorem]{Definition}

\makeatletter
\@namedef{subjclassname@2020}{
	\textup{2020} Mathematics Subject Classification}
\makeatother
\begin{document}
	\title{Algebraic methods for supersmooth spline spaces}
	\author[D. Toshniwal]{Deepesh Toshniwal} 
	\address{Deepesh Toshniwal\\
		Delft Institute of Applied Mathematics\\
		Delft University of Technology, the Netherlands}
	\email{d.toshniwal@tudelft.nl}
	\urladdr{\url{https://dtoshniwal.github.io}}
	
	\author[N. Villamizar]{Nelly Villamizar} 
	\address{Nelly Villamizar\\
		Department of Mathematics\\
		Swansea University, United Kingdom}
	\email{n.y.villamizar@swansea.ac.uk}
	\urladdr{\url{https://sites.google.com/site/nvillami}}
	\thanks{D. Toshniwal was supported by project number 212.150 funded through the Veni research programme by the Dutch Research Council (NWO)}
	\thanks{N. Villamizar was supported by the UK Engineering and Physical Sciences Research Council (EPSRC) New Investigator Award EP/V012835/1.}
	\begin{abstract}
		Multivariate piecewise polynomial functions (or splines) on polyhedral complexes have been extensively studied over the past decades and find applications in diverse areas of applied mathematics including numerical analysis, approximation theory, and computer aided geometric design.
		In this paper we address various challenges arising in the study of splines with enhanced mixed (super-)smoothness conditions at the vertices and across interior faces of the partition. 
		Such supersmoothness can be imposed but can also appear unexpectedly on certain splines depending on the geometry of the underlying polyhedral partition.
		Using algebraic tools, a generalization of the Billera-Schenck-Stillman complex that includes the effect of additional smoothness constraints leads to a construction which requires the analysis of ideals generated by products of powers of linear forms in several variables.
		Specializing to the case of planar triangulations, a combinatorial lower bound on the dimension of splines with supersmoothness at the vertices is presented, and we also show that this lower bound gives the exact dimension in high degree. The methods are further illustrated with several examples.
	\end{abstract}
	
	
	\maketitle
	\section{Introduction}\label{sec:introduction}
	A multivariate spline is a piecewise polynomial function defined on a  partition $\Delta$ of a domain in $\Omega\subseteq\R^n$ such that, as a function on $\Omega$, it is continuously differentiable up to a fixed order $r\geq 0$.
	A fairly more general definition arises when additional smoothness conditions are imposed on specific faces of the partition $\Delta$. 
	Such splines are called \emph{supersmooth splines} or \emph{supersplines}, and in this article we study them using algebraic tools.
	
	Spline spaces with supersmoothness are used for spline-based finite elements or isogeometric analysis applications \cite{hughes2005isogeometric}.
	On a general planar triangulation, the dimension of the space of $C^r$-continuous splines of polynomial degree at most $d$ may depend on the geometry of the partition for small $d$.
	This is undesirable for finite elements as it complicates, for instance, the efficient construction of locally supported basis functions.
	However, enhanced supersmoothness can be employed to eliminate this geometric-dependence and yield more tractable spline spaces; e.g., see Speleers \cite{speleers_2013} and Groselj and Speleers \cite{grovselj2021super}.
	Given this, developing an understanding or spline spaces with (enhanced) supersmoothness has both theoretical and practical relevance.
	In this article, we present an application of homological methods toward this task.
	
	Classically, splines have  been studied using Bernstein-B\'ezier representations and the construction of minimal determining sets, see \cite{lai2007spline} and the references therein.
	These methods were first applied to superspline spaces on  triangulations by Chui in \cite{chui90}, where a special order of supersmoothness $r + \lfloor (d-2r-1)/2 \rfloor$ was imposed on the vertices of the partition for $C^r$-spline spaces of degree at most $d \geq 3r+2$.
	The motivation to construct this spline space came from the construction of locally supported basis functions and optimal finite element approximation. 
	Splines with arbitrary uniform supersmoothness were introduced by Schumaker in \cite{schumaker89}; 
	and splines with varying orders of supersmoothness at the vertices by Ibrahim and Schumaker in \cite{ibrahim91}. 
	See also \cite[Chapter 5]{lai2007spline} where Bernstein-B\'ezier methods for splines on triangulation and well-known results on superspline spaces have been collected and summarized.
	Alfeld and Schumaker in \cite{alfeld03} introduced the notion of \emph{smoothness functionals} and provided lower and upper bounds for bivariate spline spaces with enhanced smoothness conditions across interior edges of the underlying triangulation.
	This led to a more general notion of supersmoothness, which can also be found in \cite[Chapter 9]{lai2007spline}. 
	
	Supersmoothness properties can be imposed but they can also appear unexpectedly on certain splines with only uniform global smoothness constraints.
	Splines with such unexpected smoothness are said to have \emph{intrinsic supersmoothness}.  
	This feature was first observed by Farin in \cite{farin80} in the case of cubic $C^1$-continuous splines on the Clough--Tocher split, which is the triangulation of a triangle with a single interior vertex and three interior edges. Farin observed that the second order derivatives of the $C^1$-splines supported in this triangulation are also continuous at the interior vertex. 
	A detailed proof of this case as well as its trivariate analog can be found in \cite{alfeld84}.
	It is now known that on a given triangulation, for certain combinations of degrees and global smoothness, the dimension of a spline space can be determined combinatorially if additional smoothness constraints on the faces of the partition are revealed and appropriate addressed.
	The latter has been studied via Bernstein--B\'ezier methods to prove results on dimension of spline spaces by Sorokina in \cite{sorokina_2010,sorokina_2013}, and Sorokina and Shektman in \cite{BorisTatyana2013}. Recently, in this direction, Floater and Hu in \cite{floater_hu_2020} determine the maximal order of intrinsic supersmoothness at vertices for various simplicial complexes with a single interior vertex.
	
	Algebraic methods developed for studying $C^r$-continuous splines \cite{billera1988homology,billeraR91,schenck1997family,Spect,schenck1997local} on polyhedral complexes were explored by Geramita and Schenck in \cite{geramita1998fat} to study spline spaces with varying order of smoothness across the codimension-1 faces of a simplicial complex in $\R^n$.
	In this approach, the connection between spline functions and fat point ideals is used to derive a dimension formula for mixed spline spaces on planar triangulations in sufficiently high polynomial degree. 
	This connection is further explored by DiPasquale in \cite{DiPasquale18} for splines on polytopal complexes, and for splines with mixed supersmoothness condition on the edges of planar quadrangular and T-meshes in \cite{mixed21,mixed20}. 
	
	The application of algebraic methods to studies of splines with mixed smoothness (i.e., with differing orders of smoothness across different codimension-1 faces of an $n$-dimensional complex) are the ones closest in spirit to the focus of this manuscript.
	We extend these algebraic methods to the setting where supersmoothness can be imposed at any arbitrary $i$-dimensional faces, $i \leq n-1$, of such a complex.
	This is a very general setting which can be used to further our understanding of both superspline and classical spline spaces.
	Indeed, the two are related by the notion of intrinsic supersmoothness, identification of which has been shown to yield a better understanding of the dimension of classical splines \cite{sorokina_2010,sorokina_2013}.
	The latter is an open problem in spline theory in general and algebraic methods have provided new results, for instance, see the recent developments in \cite{MN-PaperA,MN-PaperB,YSS20,YS19}.
	
	The paper is organized as follows. 
	In Section \ref{sec:meshes_and_splines} we set up notation, give the definition of mixed splines and superspline spaces.
	In Section \ref{sec:topology} we present the relevant homological and algebraic background to study the dimension of superspline spaces. 
	We devote Section \ref{sec:idealedges} to the study of certain ideals that that arise when considering supersmooth splines.
	We consider the special case of planar triangulations and derive a lower bound on the dimension of the superspline space in Section \ref{sec:bounds}; we prove that the lower bound coincides with the exact dimension in large degree.
	Finally, we devote Section \ref{sec:examples} to specific examples of superspline spaces that appear in the literature \cite{morgan75,chui85,speleers_2013,floater_hu_2020} before concluding.
	All examples have been computed using Macaulay2 \cite{M2}, and the code for the same can be downloaded from \url{https://github.com/dtoshniwal/M2_supersmoothness}.
	\section{Splines with mixed and supersmoothness conditions}\label{sec:meshes_and_splines}
	In this section we set notation and important definitions concerning the spline spaces that we will study in the rest of the paper.
	
	We denote by $\Delta$ a simplicial complex embedded in $\R^n$.
	If there is no confusion about the embedding, we identify $\Delta$ with its embedding and write $\Delta\subseteq \R^n$.
	If $n=2$ we refer to $\Delta$ as a triangulation, and as a tetrahedral complex if
	$n=3$.
	We write $\Delta^\circ$ and $\partial\Delta=\Delta\setminus \Delta^\circ$ for the collection of interior and boundary faces of $\Delta$, respectively.
	The set of $i$-dimensional faces of $\Delta$, also called $i$-faces, is denoted $\Delta_i$, and $\Delta_i^\circ\subseteq\Delta_i$ is the set of the interior $i$-faces, for $i=0,\dots, n-1$.
	The number of elements of $\Delta_i$ and $\Delta_i^\circ$ is denoted $f_i$ and $f_i^\circ$, respectively.
	
	Denote by $\tR = \R[x_1,\dots,x_n]$ the polynomial ring in $n$-variables, and by $\tR_{\leq d}$ the vector space of polynomials in $\tR$ of total degree at most $d$.
	If $\Delta\subseteq\R^n$ is a simplicial complex, we write $C^r(\Delta)$ for set of all functions $F\colon \Delta\rightarrow\R $ which are continuously differentiable of order $r$ on $\Delta$. We call these functions $C^r$-continuous, or $C^r$-smooth, on $\Delta$.
	
	\begin{definition}\label{def:splines}
		Let $\Delta\subseteq \R^n$ be a simplicial complex, and $0\leq r\leq d$ be integers. 
		The set $\calS_d^r(\Delta)$ of $C^r$-continuous splines on $\Delta$ is defined as the set of all piecewise polynomial functions on $\Delta$ of degree at most $d$ that are continuously differentiable up to order $r$ on $\Delta$. More precisely, 
		\[\calS^{r}_{d}(\Delta)= \bigl\{ f\in C^r(\Delta)\colon f|_{\sigma}\in\tR_{\leq d} \text{\ for all \ } \sigma\in\Delta_n\bigr\}.\] 
		If $f\in\calS^{r}_{d}(\Delta)$ we say that $f$ is a $C^r$-spline, or a $C^r$-continuous (or -smooth) spline, on $\Delta$.
		The collection of all $C^r$-splines on $\Delta$ is denoted $\calS^{r}(\Delta)=\bigcup_{d\geq 0}S_d^r(\Delta)$. 
	\end{definition}
	For a given simplicial complex $\Delta$, we extend Definition \ref{def:splines} and consider spline functions with variable smoothness conditions at the vertices or across the interior faces of $\Delta$.
	If $\beta\in\Delta_i^\circ$, let us denote by $\Delta_\beta$ the \emph{star of $\beta$ in $\Delta$}, that is the simplicial complex composed by all the simplices of $\Delta$ having $\beta$ as one of their faces.
	Following the notation in \cite{lai2007spline} and \cite{geramita1998fat}, we first define the space of splines with \emph{mixed smoothness} conditions across the interior codimension-1 faces  of $\Delta$.
	\begin{definition}[Spline functions with mixed smoothness] 
		For a simplicial complex $\Delta\subset \R^n$, and a non-negative integer $d$, let $\bm r= \{r_\beta\colon {\beta\in\Delta_{n-1}^\circ}\bigr\}$ be a set of integers $0\leq r_\beta\leq d$ associated to the codimension $1$-faces $\beta\in\Delta_{n-1}^\circ$. 
		The space $\calS^{\bm r}_{d}(\Delta)$ of splines with \emph{mixed smoothness} $\bm r$ on $\Delta$ is defined as the set of all $C^0$-continuous functions on $\Delta$ which are splines with smoothness $r_\beta$ across the face $\beta$ for each $\beta\in\Delta_{n-1}^\circ$. Namely,
		\begin{equation*}
			\calS^{\bm r}_{d}(\Delta) = 
			\bigl\{f\in C^0(\Delta) \colon f|_{\Delta_\beta}\in \calS^{r_\beta}_d(\Delta_\beta) \text{\ for all \ } \beta \in \Delta_{n-1}^\circ\bigr\},
		\end{equation*}	
		where $\Delta_\beta$ is the star of the face $\beta$ in $\Delta$.
		Similarly as before, we denote $\calS^{\bm r}(\Delta)=\bigcup_{d\geq 0}S_d^{\bm r}(\Delta)$. 
		If $r_\beta=r\in\Z_{\geq 0}$ for all $\beta\in\Delta_{n-1}^\circ$, then $\calS^{\bm r}_{d}(\Delta)$ coincides with $\calS_d^r(\Delta)$ in Definition \ref{def:splines}. 
		In this case we write $\calS^{\bm r}_{d}(\Delta)=\calS^{r}_{d}(\Delta)$. 
	\end{definition}
	We now define spline functions with variable order of smoothness at the $i$-faces $\Delta$ for $i=0,\dots, n-2$. 
	We follow the notation in \cite{lai2007spline} for planar $\Delta$ and call the sets of these functions \emph{superspline spaces}. 
	We say that a spline $f\in \calS^{0}_{d}(\Delta)$ is $C^s$-continuous at a face $\beta\in\Delta_i$ provided that, for all $\sigma\in\Delta_n$ such that $\beta$ is a face of $\sigma$, all polynomials $f|_\sigma$ have common derivatives up to order $s$ on $\beta\subseteq\R^n$.
	In this case we say that $f$ have \emph{supersmoothness} $s$ at $\beta$ and write $f|_{\Delta_{\beta}}\in C^s(\beta)$, or simply  $f\in C^s(\beta)$.
	\begin{definition}[Superspline functions]\label{def:supersplines}
		Suppose $\Delta\subseteq \R^n$ is a simplicial complex, and $r_\tau$ and $d$ are integers such that $0\leq r_\tau\leq d$ for each $\tau\in\Delta_{n-1}^\circ$.
		For a fixed $0\leq i\leq n-2$, let $\bm s=\bigl\{s_\beta\colon \beta\in\Delta_{i}\bigr\}$ be a sequence of integers $s_\beta$ with $0\leq s_\beta\leq d$. 
		The \emph{superspline space}  $\calS^{\br,\bs}_{d}(\Delta)$ is defined as the set of all $C^r$-continuous splines on $\Delta$ with supersmoothness $s_\beta$ at $\beta$ for each face $\beta\in\Delta_{i}$ i.e., 
		\begin{equation*}
			\calS^{\br, \bs}_{d}(\Delta) = 
			\bigl\{f\in \calS^{\br}_{d}(\Delta) \colon f\in C^{s_\beta}(\beta) \text{\ for all \ } \beta \in \Delta_{i}\bigr\}.
		\end{equation*}	
		We denote $\calS^{\br,\bm s}(\Delta)=\bigcup_{d\geq 0}S_d^{\br,\bs}(\Delta)$. 
		If $s_\beta=s\in\Z_{\geq 0}$ for all $\beta\in\Delta_{i}$, we write $\calS^{\br,\bs}_{d}(\Delta)= \calS^{\br, s}_{d}(\Delta)$; if $r_\tau=r\in\Z_{\geq 0}$ for all $\tau\in\Delta_{n-1}^\circ$ we simply write  $\calS^{r,\bs}_{d}(\Delta)$; in the case $s=r$ then $\calS^{\br,\bs}_{d}(\Delta)=\calS^{r}_{d}(\Delta)$.
	\end{definition} 
	\begin{remark}
		Notice that if  $\gamma\in\Delta_i$ for  $0\leq i\leq n-2$ is a face of $\beta\in\Delta_{n-1}$ and $f\in \calS^{r}_{d}(\Delta)$, then $f\in C^s(\gamma)$ does not necessarily imply  $f|_{\Delta_{\beta}}\in \calS^{s}_d(\Delta_\beta)$. 
		Conversely, if $f|_{\Delta_{\beta}}\in \calS^{s}_d(\Delta_\beta) $ holds for all $(n-1)$-face $\beta\in \Delta_\gamma$ then  $f\in C^s(\gamma)$ for each face $\gamma\subseteq \beta$.
	\end{remark}
	If we fix an index $0\leq i\leq n-2$, and 
	consider only superspline functions that posses the same order of enhanced smoothness at the $i$-faces of the simplicial complex. In this setting, in the case $n=2$, the only superspline space will be that of splines with uniform supersmoothness at each vertex of the triangulation. In the case $n=3$, we can consider two superpline spaces, one composed of splines with supersmoothness across the edges and the other of splines with enhanced supesmothness at the vertices of the given tetrahedral partition. 
	\section{Supersplines as the homology of a chain complex}\label{sec:topology}
	In this section we review the necessary results from \cite{billeraR91,billera1988homology,geramita1998fat,schenck1997local}, and extend these results to the setting of superspline spaces $\calS^{\br,\bs}(\Delta)$ introduced in Section \ref{sec:meshes_and_splines}. 
	
	First we recall that for any pair of integers $r,d\geq 0$, the study of the splines $\calS^{r}_d(\Delta)$ on $\Delta$ of degree at most $d$ and global smoothness $r$ can be reduced to the study of splines on a simplicial complex  whose polynomial pieces are homogeneous polynomials of degree $d$.
	
	In fact, if $\Delta\subseteq \R^n$ is the star of a vertex (i.e., if all simplices in $\Delta$ share a common vertex), then 
	\begin{equation}\label{eq:starvertexhom}
		\calS^r(\Delta)\cong \bigoplus\limits_{i\ge 0} \calS^r(\Delta)_i, \mbox{ and}\quad \calS^r_d(\Delta)\cong \bigoplus\limits_{i=0}^d \calS^r(\Delta)_i,
	\end{equation}
	where $\calS^r(\Delta)_i$ denotes the splines on $\Delta$ of degree exactly $i$, and the isomorphism is as $\R$-vector spaces.  
	
	If $\Delta\subseteq$ is not the star of a vertex, then the isomorphism \eqref{eq:starvertexhom} does not hold for $\calS^r(\Delta)$, but 
	one can associate to $\Delta$ a star of a vertex $\hat{\Delta}\subseteq \R^{n+1}$ and~\eqref{eq:starvertexhom} will still be valid.
	This new complex $\hat{\Delta}$ can be constructed as follows. If $x_1,\dots,x_n$ are the coordinates of $\R^n$, consider the embedding $\phi\colon\R^n\rightarrow \R^{n+1}$ in the hyperplane $\{x_{0}=1\}\subseteq \R^{n+1}$ given by  $\phi(x_1,\dots, x_n)=(1,x_1,\dots, x_n)$.
	If $\sigma$ is a simplex in $\R^n$, the \textit{cone over} $\sigma$, denoted $\hat{\sigma}$, is the simplex in $\R^{n+1}$ which is the convex hull of the origin in $\R^{n+1}$ and $\phi(\sigma)$.  
	If $\Delta\subseteq\R^n$ is a simplicial complex, the \textit{cone over} $\Delta$, denoted $\hat \Delta$, is the simplicial complex consisting of the simplices $\bigl\{\hat{\beta}:\beta\in\Delta\bigr\}$ along with the origin in $\R^{n}$. 
	Then, by construction, $\hat\Delta\subseteq\R^{n+1}$ is the star of the origin and ~\eqref{eq:starvertexhom} yields
	$
	\calS^r(\hat\Delta)\cong \bigoplus\limits_{i\ge 0} \calS^r(\hat \Delta)_i\mbox{ and } \calS^r_d(\hat\Delta)\cong \bigoplus\limits_{i=0}^d \calS^r(\hat \Delta)_i.
	$
	The following result from Billera and Rose \cite{billeraR91} links these two spline spaces.
	\begin{theorem}{\normalfont \cite[Theorem~2.6]{billeraR91}}\label{thm:Homogenize}
		If $\Delta\subseteq \R^n$ is a simplicial complex and $\hat \Delta$ is the cone over $\Delta$ in $\R^{n+1}$ then	$\calS^r_d(\Delta)\cong \calS^r(\hat\Delta)_d$.
	\end{theorem}
	In the following we extend Theorem \ref{thm:Homogenize} to the superspline functions introduced in Definition \ref{def:supersplines}. 
	
	\subsection{Superspline ideals}
	Suppose $\Delta\subseteq\R^n$ is an $n$-dimensional simplicial complex.
	As defined above, let $\hat \Delta\subseteq \R^{n+1}$ be the cone over $\Delta$, and denote by $\tS=\R[x_0,x_1,\dots,x_n]$ the polynomial ring associated to $\hat\Delta$.
	The homogenization of a polynomial $f\in\tR$ in $\tS$ is denoted $\hat{f}$.
	Conversely, if $f \in\tS$, its dehomogenized taking $x_0=1$ is denoted $\check{f} \in \tR$.
	
	For homogeneous polynomials $f_1,\dots,f_k\in \tS$, we denote by $\langle f_i\rangle\subseteq \tS$ the ideal generated by $f_i$ and $\langle f_1,\dots, f_k\rangle=\sum_{i=1}^{k}\langle f_i\rangle$ the ideal of $\tS$ generated by $f_1,\dots, f_k$. 
	We write $\V(f_1,\dots, f_k)\subseteq \R^{n+1}$ for the set of points $\bm p\in\R^{n+1}$ such that $f_i(\bm p)=0$ for all $i=1,\dots, k$.
	Similarly, we define $\V(\check f_1,\dots, \check f_k)\subseteq \R^n$ for $\check f_i\in\tR$.
	
	Fix $0\leq i\leq n-2$, and take two sets of integers $\bm r= \{r_\tau\colon \tau\in\Delta_{n-1}^\circ\bigr\}$ and $\bm s=\bigl\{s_\beta\colon \beta\in\Delta_{i}\bigr\}$ such that $r_\tau, s_\beta\geq 0$ for each $\tau\in\Delta_{n-1}^\circ$ and $\beta\in\Delta_i$. 
	To each face of $\Delta$ we associate an (homogeneous) ideal in $\tS$ as follows.
	\begin{itemize}
		\renewcommand{\labelitemi}{\scalebox{0.5}{$\blacksquare$}}	
		\item If $\sigma\in\Delta_n$ define $\tJ(\sigma)=0$.
		\item If $\tau\in\Delta_{n-1}^\circ$, let $\ell_\tau\in\tS$ be (a choice of) a linear form vanishing on $\hat\tau$.
		For each $i$-face $\beta\subset\tau$ let $\hat\fm_{\beta}=\bigl\{\hat f\in \tS\colon f\in\fm_{\beta}\bigr\}$
		where $\fm_\beta\subseteq\tR$ is the ideal of all polynomials vanishing at $\beta$.
		We define 
		\begin{equation}\label{eq:edgesgen}
			\tJ(\tau)=\bigl\langle \ell_\tau^{r_\tau+1}\rangle\bigcap\bigl\{\hat\fm_{\beta}^{s_\beta+1}\colon \beta\in\Delta_i\, , \beta\subset\tau\bigr\}\,.
		\end{equation}
		\item If $\gamma\in\Delta_{j}$ for $0\leq j\leq n-2$, take 
		\begin{equation}\label{eq:ifacesgen}
			\textstyle{	\tJ(\gamma)=\sum_{\tau\ni\gamma,\;\tau\in\Delta_{n-1}^\circ} \tJ(\tau)\,.}
		\end{equation}
	\end{itemize}
	Additionally, we denote by $\check\tJ(\tau)$ the ideal in $\tR$ corresponding to the edge $\tau\in\Delta_1^\circ$, namely
	\begin{equation}\label{eq:affedgegen}
		\check{\tJ}(\tau)=\bigl\langle\check\ell_\tau^{r_\tau+1}\bigr\rangle\bigcap\bigl\{\fm_{\beta}^{s_\beta+1}\colon  \,\beta\in\Delta_{i},\; \beta\subset\tau\bigr\}, 
	\end{equation}
	where $\check{\ell}_{{\tau}}\in\tR$ is a linear polynomial vanishing at $\tau$, and the ideal $\fm_\beta\subseteq\tR$ is the ideal of all polynomials vanishing at $\beta$. 
	The ideal $\tJ(\tau)$ can be defined as the homogenization of $\check\tJ(\tau)$ in $\tS$.
	
	Note that, if  $r=s$, the ideal $\tJ(\tau)$ associated to $\tau\in\Delta_{n-1}^\circ$ reduces to $\tJ(\tau)=\langle \ell_{\tau}^{r+1}\rangle$, and we recover the ideals defined by Schenck and Stillman in \cite{schenck1997local}.
	
	\begin{remark}[$\Delta\subseteq \R^2$]
		If $n=2$, for simplicity we write $\tR=\R[x,y]$ and $\tS=\R[x,y,z]$ where $x,y$ are the coordinates of $\R^2$ (and hence of the simplicial complex $\Delta$). 
		Suppose $r_\tau=r$ for all edges $\tau$ and and $s_\gamma=s$ for every vertex $\gamma$ of $\Delta$. 
		The ideal $\tJ(\tau)$ for an edge $\tau\in\Delta_1$ defined in Equation \eqref{eq:edgesgen} can be rewritten as
		\begin{equation}\label{eq:edges}
			\tJ(\tau)=\bigl\langle\ell^{r+1}_\tau \bigr\rangle\cap \hat\fm^{s+1}_\gamma\cap\hat{\fm}_{\gamma'}^{s+1}=\bigl\langle \ell_\tau^{s+1-i}\ell_{\tau,\gamma}^i\ell_{\tau,\gamma'}^ i\colon 0\leq i\leq s-r\bigr\rangle\,,
		\end{equation}	
		where $\tau=[\gamma,\gamma']$ is the edge with vertices $\gamma$ and $\gamma'$; the linear form $\ell_\tau\in\tS$ vanishes on $\hat\tau$, and $\ell_{\tau,\gamma}$ and $\ell_{\tau,\gamma'}$ are (a choice of) linear forms in $\tS$ such that $\V(\ell_\tau,\ell_{\tau,\gamma})$ and $\V(\ell_\tau,\ell_{\tau,\gamma'})$ are the lines containing the faces $\hat\gamma$ and $\hat\gamma'$ of $\hat{\Delta}$, respectively.
		
		If $\check{\ell}_{\tau}$ and $\check{\ell}_{\tau,\gamma}$ are the  dehomogenizations by taking $z=1$ of the forms ${\ell}_{\tau}$ and ${\ell}_{\tau,\gamma}$ then $\{\gamma\}=\V(\check{\ell}_{\tau},\check{\ell}_{\tau,\gamma})$, where the latter denotes the affine variety of the two linear polynomials in $\tR$ vanishing at $\gamma$. In this case, Equation \eqref{eq:affedgegen} reduces to
		\begin{equation*}
			\check\tJ(\tau)=\bigl\langle \check\ell_\tau^{r+1}\bigl\rangle\cap\fm_\gamma^{s+1}\cap \fm_{\gamma'}^{s+1},
		\end{equation*} where $\check\ell_{\tau}$ is a linear polynomial vanishing at the edge $\tau$, and $\fm_\gamma$ and $\fm_{\gamma'}$ are the (maximal) ideals in $\tR$ of all polynomials vanishing at the points $\gamma$ and $\gamma'$, respectively. 
		
		Notice that, the ideal $\tJ(\tau)$ as defined in Equation \eqref{eq:edges} is in fact independent of the choice of the linear forms ${\ell}_{\tau,\gamma}$ and ${\ell}_{\tau,\gamma'}$.
		Namely, we can choose $\ell_\tau$ and ${\ell}_{\tau,\gamma}$ as the generators of $\hat\fm_\gamma=\bigl\langle \ell_{\tau},\ell_{\tau,\gamma}\bigr\rangle$, and similarly $\hat\fm_{\gamma'} = \bigl\langle\ell_{\tau},\ell_{\tau,\gamma'}\bigr\rangle$. 
		If $\ell_{\tau}$, $\ell_{\tau,\gamma}$ and $\ell$ are three distinct linear forms vanishing at $\hat{\gamma}$, then it is easy to see that $\ell$ can be written as a linear combination $\ell=a\ell_{\tau}+b\ell_{\tau,\gamma}$, for $a,b\in\R$. 
		A generator of the ideal $\hat\fm_\gamma^k=\bigl\langle \ell_{\tau},\ell_{\tau,\gamma}\bigr\rangle^k$, for some $k\geq 1$, has the form $\ell_{\tau}^i\ell_{\tau,\gamma}^j$, with $i+j=k$, and $\ell_{\tau}^i\ell^j=\ell_{\tau}^i(a\ell_\tau+b\ell_{\tau,\gamma})^j$, which is clearly an element of $\hat\fm_\gamma^k$. Hence $\langle\ell_{\tau},\ell\bigr\rangle^k\subseteq\hat\fm_\gamma^k$, and the converse trivially follows writing $\ell_{\tau,\gamma}$ in terms of $\ell_\tau$ and $\ell$. A similar argument shows the corresponding statement for the ideal $\hat\fm_{\gamma'}$.
	\end{remark}
	
	\subsection{A chain complex of supersplines}
	Recall that a simplicial complex $\Delta\subseteq \R^n$ is \emph{pure} if all its maximal faces (with respect to inclusion) are of dimension $n$; and it is \emph{hereditary} if for all pair of faces $\sigma,\sigma'\in\Delta_n$ such that $\sigma\cap\sigma'=\beta\in\Delta_i$ there is a sequence of $n$-faces $\sigma=\sigma_0,\sigma_1,\dots,\sigma_{m-1},\sigma_m=\sigma'$such that $\beta\in\sigma_i$ for all $i$ and  $\sigma_{i-1}\cap\sigma_{i}\in\Delta_{n-1}^\circ$ for each $i=1, \dots, m$. 
	
	If $\Delta$ is a \emph{pure} and \emph{hereditary} $n$-dimensional simplicial complex, Billera proved in \cite{billera1988homology} the following algebraic criterion for a piecewise polynomial function on a simplicial complex $\Delta\subseteq \R^n$ to be $C^r$-continuous on $\Delta$.
	
	\begin{theorem}{\normalfont \cite[Theorem 2.4]{billera1988homology}}\label{thm:algcriterion}
		Suppose $\Delta\subseteq \R^n$ is a pure and hereditary simplicial complex and $r\geq 0$ is an integer. Then $f\in\calS^r(\Delta)$ if and only if $\hat f|_{\hat\sigma}-\hat f|_{\hat\sigma'}\in \tJ(\tau)$ or, equivalently, if and only if $f|_{\sigma}-f|_{\sigma'}\in \langle \check\ell_\tau^{r+1}\rangle$, for every pair $\sigma,\sigma'\in\Delta_n$ satisfying $\sigma\cap\sigma'=\tau\in\Delta_{n-1}^\circ$.
	\end{theorem}
	\begin{remark}
		In the case $\Delta\subseteq\R^2$, Wang in \cite{wang75} and Chui in \cite[Theorem 4.2]{chui88} provided earlier proofs of Theorem \ref{thm:algcriterion}.
	\end{remark}
	In the following, we prove an analogous criterion to Theorem \ref{thm:algcriterion} for splines with smoothness $\br$ across the codimension-$1$ faces and supersmoothness $\bs$ at all the $i$-faces of the partition.  
	\begin{theorem}\label{prop:newalgcriteriongen}
		Suppose $\Delta\subseteq \R^n$ is a pure and hereditary simplicial complex and $\calS^{\br,\bs}(\Delta)$ denotes the set of splines with smoothness  $\br=\{r_\tau\colon\tau\in\Delta_{n-1}^\circ\}$ at the codimension 1-faces and supersmoothness $\bs=\{s_\beta\colon \beta\in\Delta_i\}$ across all the $i$-faces of $\Delta$, for a fixed $0\leq i\leq n-2$.
		Then $f\in\calS^{\br,\bs}(\Delta)$ if and only if  $\hat{f}|_{\hat{\sigma}}-\hat{f}|_{\hat{\sigma}'}\in\tJ(\tau)$ or, equivalently, if and only if $f|_{\sigma}-f|_{\sigma'}\in \check\tJ(\tau)$, for all $\tau\in\Delta_{n-1}^\circ$ and $\sigma,\sigma'\in\Delta_n$ satisfying $\sigma\cap\sigma'=\tau$. 
	\end{theorem}
	\begin{proof}
		Let $\sigma,\sigma'\in\Delta_n$ such that $\sigma\cap\sigma'=\tau\in\Delta_{n-1}^\circ$.
		Suppose $f|_{\sigma}-f|_{\sigma'}\in\check\tJ(\tau)$. 
		In particular, $f|_{\sigma}-f|_{\sigma'}\in \bigl\langle \check\ell_\tau^{r_\tau+1}\rangle$ and clearly the restriction of the derivatives up to order $r_\tau$ of $f|_{\hat\sigma}-f|_{\hat\sigma'}$ to the edge $\hat\tau$ are zero.
		On the other hand, $f|_{\sigma}-f|_{\sigma'}\in\fm_{\beta}^{s_\beta+1} $ for each  $\beta\in\Delta_{i}$ such that $\beta\subset\tau$, so the polynomial $f|_{\sigma}-f|_{\sigma'}$, and all its derivatives up to order $s_\beta$, vanish at $\beta$.
		
		By hypothesis $\Delta$ is hereditary, then there is a sequence of $n$-faces $\sigma_0,\sigma_1,\dots,\sigma_m$ such that $\sigma_j\supset\beta$ for all $j$ and  $\sigma_{j-1}\cap\sigma_{j}\in\Delta_{n-1}^\circ$. 
		Applying the previous argument to each pair of faces $\sigma_{j-1}$ and $\sigma_{j}$, we get that all the derivatives up to order $s_\beta$ of $f|_{\sigma_{j-1}}$ and $f|_{\sigma_{j}}$ coincide at $\beta$ for every $j=1,\dots, m$, and hence $f\in C^{s_\beta}(\beta)$ for each $\beta\in\Delta_{i}$. It follows that $f\in \calS^{\br,\bs}(\Delta)$.
		
		Conversely, if $f\in\calS^{\br,\bs}(\Delta)$ then by Theorem \ref{thm:algcriterion} $f|_{\sigma}-f|_{\sigma'}\in\bigl\langle\check\ell_\tau^{r_\tau+1}\bigr\rangle$ for all $\tau\in\Delta_{n-1}^\circ$. 
		Let $\beta$ be one of the $i$-faces of $\tau$. 
		The ideal $\fm_{\beta}=\{g\in\tR\colon g(\beta)=0\}$ is generated by 
		$n-i$ linearly independent linear polynomials, each of them vanishing at $\beta$. 
		By hypothesis, the function $f|_{\sigma}-f|_{\sigma'}$, and all its derivatives up to order $s_\beta$, are zero when restricted to $\beta$. 
		If follows $f|_{\sigma}-f|_{\sigma'}\in\fm_{\beta}$, and 
		by induction (on the order of the derivatives) we get that  $f|_{\sigma}-f|_{\sigma'}\in \fm_{\beta}^{s_\beta+1}$.
		This argument applies to every $i$-face $\beta\subset\tau$ and leads to $f|_{\sigma}-f|_{\sigma'}\in\check\tJ(\tau)$ for each $\tau\in\Delta_{n-1}^\circ$, as required.
	\end{proof}
	We now extends the construction by Billera \cite{billera1988homology} and refined by Schenck and Stillman in \cite{schenck1997local} to the context of superspline spaces. 
	
	If $\Delta\subseteq\R^n$ is a simplicial complex, let $\bigoplus_{\beta\in\Delta_i}\tS[\beta]$ be the direct sum of the polynomial ring $\tS$, where $[\beta]$ is a formal basis symbol corresponding to the $i$-face $\beta$. 
	If $\partial_i$ is the simplicial boundary map relative to the boundary $\partial\Delta$, we denote by $\calR$ the chain complex 
	\[
	\calR\colon 0\rightarrow \bigoplus_{\beta\in\Delta_n}\tS[\beta]\xrightarrow{\partial_n}\dots\rightarrow \bigoplus_{\sigma\in\Delta_2}\tS[\sigma] \xrightarrow{\partial_2} \bigoplus_{\tau\in\Delta_1^\circ}\tS[\tau] \xrightarrow{\partial_1} \bigoplus_{\gamma\in\Delta_0^\circ}\tS[\gamma] \rightarrow 0.\]
	The restriction of the maps $\partial_i$ to the ideals $\bigoplus_{\beta\in\Delta_i}\tJ(\beta)$ yields the subcomplex $\calJ$ given by 
	\begin{equation}\label{eq:complexJ}
		\calJ\colon 0\rightarrow  \bigoplus_{\tau\in\Delta_{n-1}^\circ}\tJ(\tau) \xrightarrow{{\partial}_{n-1}}\cdots\xrightarrow{\partial_1} \bigoplus_{\gamma\in\Delta_0^\circ}\tJ(\gamma) \rightarrow 0; 
	\end{equation}
	and taking the quotient leads to chain complex $\calR/\calJ$ given by
	\begin{equation}\label{eq:quotient}
		\calR/\calJ\colon 0\rightarrow \bigoplus_{\beta\in\Delta_n}\tS[\beta] \xrightarrow{\overline{\partial}_n}\cdots\rightarrow \bigoplus_{\tau\in\Delta_1^\circ}\tS/\tJ(\tau) \xrightarrow{\overline{\partial}_1} \bigoplus_{\gamma\in\Delta_0^\circ}\tS/\tJ(\gamma) \rightarrow 0.
	\end{equation}
	If we take $s_\beta=r_\tau=r$, for some $r\in\Z_{\geq 0}$, for all faces $i$-faces $\beta$ and all codimension-1 faces $\tau$, the complex $\calR/\calJ$ reduces to that in \cite{schenck1997local}.
	
	We recall that for a chain complex $\calC$ with boundary maps $\partial_i$, the $i$-th homology module $H_i(\calC)$ is defined as   $H_i(\calC)=\ker(\partial_{i})/\im(\partial_{i-1})$.
	It was shown by  Billera  in \cite{billera1988homology} that 
	$\calS^{r}(\hat\Delta)\cong H_n(\calR/\calJ)=\ker{\overline{\partial}_n}\,$.
	This isomorphism also holds in our settings, it follows by the algebraic criterion in Theorem \ref{prop:newalgcriteriongen}. 
	However, in contrast to the case of splines with uniform global smoothness conditions $r=s$, in our settings we need to specify the superspline space we consider on $\Delta$ and the corresponding one on $\hat\Delta$. 
	Namely, if we take the set $\calS^{\br,\bs}(\Delta)$ of $C^r$-continuous splines on $\Delta$ with supersmoothness $\bs$ on the $i$-faces $\beta\in\Delta_i$, the corresponding spline space on $\hat\Delta$, denoted $\calS^{\br,\bs}(\hat\Delta)$, is the set of $C^r$-splines on $\hat\Delta$ with supersmoothness $\bs$ at the $(i+1)$-faces $\hat\beta$ of $\hat\Delta$.
	Following this notation we have the following results.
	\begin{corollary}\label{cor:kernelcomplexR/J}
		Let $\Delta\subseteq \R^n$ be a pure and hereditary simplicial complex and let $0\leq r_\tau\leq s_\beta$ be integers for each $\tau\in\Delta_{n-1}^\circ$ and $\beta\in\Delta_i^\circ$, for a fixed $0\leq i\leq n-2$. 
		Then, $\calS^{\br,\bs}(\hat\Delta)\cong \ker\bigl(\overline{\partial}_n\bigr)$, where $\calS^{\br,\bs}(\hat\Delta)$ is the set of $C^r$-splines with supersmoothness $\bs$ at the $(i+1)$-faces $\hat\beta$, and $\overline{\partial}_n$ is the differential map in the chain complex $\calR/\calJ$ in Equation \eqref{eq:quotient}.
	\end{corollary}
	\begin{proof}
		By Theorem \ref{prop:newalgcriteriongen}, we have that $f\in\calS^{\br,\bs}(\hat\Delta)$ if and only if $\partial_n(f)|_\tau=f|_{\hat\sigma}-f|_{\hat\sigma'}\in\tJ(\tau)$ for each $\tau\in\Delta_{n-1}^\circ$, or equivalently, if and only if  $f\in\ker(\overline{\partial}_n)$, as required.
	\end{proof}
	\begin{proposition}\label{thm:spline_space_isomorphism}
		If $\Delta\subseteq \R^n$ is a pure and hereditary simplicial complex, then $ \calS^{\br,\bs}_{d}(\Delta)\cong \calS^{\br,\bs}(\hat\Delta)_d\,,$ as real vector spaces.	
	\end{proposition}
	\begin{proof}
		We follow the argument used to prove the corresponding statement for  $\calS^{r}_{d}(\Delta)$ in \cite[Theorem 2.6]{billeraR91}. 
		We define the map 
		$\varphi\colon\calS^{\br,\bs}(\hat\Delta)_d\rightarrow\calS^{\br,\bs}_{d}(\Delta)$ by $\varphi(f)|_{\hat\sigma}= {\check f}|_{\sigma}$ for each $\sigma\in\Delta_n$, where $\check f|_{\sigma}$ is the dehomogenization of $f|_\sigma$ and $\hat\sigma$ is the cone over $\sigma$.
		It is easy to see that $\varphi$ is an $\R$-linear map. 
		Theorem \ref{prop:newalgcriteriongen} applied to both $\calS^{\br,\bs}(\Delta)$ (with supersmoothness $\bs$ at the $i$-faces $\beta$ of $\Delta$) and $\calS^{\br,\bs}(\hat\Delta)$ (with supersmothness $\bs$ at the $(i+1)$-faces $\hat\beta$ of $\hat\Delta$) implies that $\varphi$ is an isomorphism of real vector spaces.
	\end{proof}
	If $\calC\colon 0\rightarrow C_n\xrightarrow{\partial_n} C_{n-1}\xrightarrow{\partial_{n-1}} \cdots \xrightarrow{\partial_1} C_0\rightarrow 0$ is a chain complex of graded modules $C_i$ and boundary maps $\partial_i$, we denote by $\chi(\calC,d)=\sum_{i=0}^n (-1)^{i}\dim (C_{n-i})_d$ the \emph{Euler-Poincar\'e characteristic} of $\calC$ at degree $d$. 
	Taking the homology modules $H_i(\calC)$ it follows $\chi(\calC,d)= \sum_{i=0}^n(-1)^{i}\dim H_{n-i}(\calC)_d$, and therefore $\sum_{i=0}^n (-1)^{i}\dim (C_{n-i})_d= \sum_{i=0}^n(-1)^{i}\dim H_{n-i}(\calC)_d$. 
	(This result from homological algebra can be found in \cite[\S 4]{spanier}, for instance.)
	We apply this equality to the complex $\calR/\calJ$, which together to Corollary \ref{cor:kernelcomplexR/J} and Proposition \ref{thm:spline_space_isomorphism}, leads to
	\begin{equation}\label{eq:dimformula}
		\dim \calS_d^{\br,\bs}(\Delta)=\sum_{i=0}^n(-1)^{i}\dim  \bigoplus_{\beta\in\Delta_{n-i}^\circ}\tS/\tJ(\beta)_d -\sum_{i=1}^n(-1)^{i}\dim H_{n-i}(\calR/\calJ)_d \,. 
	\end{equation}
	In Equation \eqref{eq:dimformula}, we consider all maximal $n$-faces of $\Delta$ to be interior, so $\Delta_n^\circ = \Delta_n$.
	\section{Supersmooth ideals at edges and vertices}\label{sec:idealedges}
	In this section we assume $\Delta$ is a simplicial complex in $\R^2$, and study the dimension of the modules on the right hand side of Equation \eqref{eq:dimformula}. 
	The objective is to get an explicit formula for  $\dim \calS_d^{\br,\bs}(\Delta)$ for special cases of $\Delta$, which we use in Section \ref{sec:bounds} to prove a lower bound on $\dim \calS_d^{\br,\bs}(\Delta)$ for arbitrary triangulations homeomorphic to a disk.
	
	If $\Delta\subseteq \R^2$, Equation \eqref{eq:dimformula} simplifies to 
	\begin{multline}\label{eq:dimformulaR2}
		\dim \calS_d^{\br,\bs}(\Delta)=\dim  \bigoplus_{\sigma\in\Delta_2} \tS[\sigma]_d - \dim \bigoplus_{\tau\in\Delta_1^\circ}\tS/\tJ(\tau)_d + \dim \bigoplus_{\gamma\in\Delta_0^\circ}\tS/\tJ(\gamma)_d  \\
		+\dim H_1(\calR/\calJ)_d 
		- \dim H_0(\calR/\calJ)_d	
	\end{multline}
	The short exact sequence of complexes $0\rightarrow \calJ\rightarrow \calR\rightarrow \calR/\calJ\rightarrow 0$ leads to long exact sequence of homology modules $H_i(\calJ)$, $H_i(\calR)$ and $H_i(\calR/\calJ)$. 
	In particular, if $\Delta\subseteq\R^2$ is homeomorphic to a disk then $H_0(\calR)=H_1(\calR)=0$ and it implies $H_0(\calR/\calJ)=0$ and $H_1(\calR/\calJ)\cong H_0(\calJ)$, respectively. Therefore, in this case
	Equation \eqref{eq:dimformulaR2} can be written as
	\begin{equation}\label{eq:dimformuladisk}
		\dim \calS_d^{\br,\bs}(\Delta)= \binom{d+2}{2} + \sum_{\tau\in\Delta_1^\circ}\dim \tJ(\tau)_d - \sum_{\gamma\in\Delta_0^\circ}\dim \tJ(\gamma)_d  +\dim H_0(\calJ)_d\, .
	\end{equation}
	
	\subsection{Ideals of edges and vertices}\label{sec:edge-vertex-ideals}
	If $\tau = [\gamma,\gamma']\in\Delta_1^\circ$ is an interior edge of $\Delta$ with vertices $\gamma$ and $\gamma'$, we write $\tJ(\tau)\subseteq\tS$ for the ideal of $\tau$ defined in \eqref{eq:edges}.  
	In the following we prove a dimension formula for the graded pieces $\tJ(\tau)_d$.
	Let $r_\tau=r\geq 0$ be the smoothness across the edge $\tau$ and take an integer $s\geq r$. 
	
	We consider two cases: In Lemma \ref{lem:dim_edge_ideal} the supersmoothness at both vertices of $\tau$ is the same i.e., $s_\gamma = s_{\gamma'} = s$, and in Lemma \ref{lem:dimEborder} only one of the vertices has supersmoothness  $s_\gamma = s$ while the smoothness at $\gamma'$ is $s_{\gamma'} = r$.
	Only the fist case is needed if the splines have supersmoothness $s$ at all the vertices of the $\Delta$, but the second is necessary for instance, to consider splines with supersmoothness only at the interior (and not at the boundary) vertices of the partition.
	
	In the following lemmas, and throughout this paper, we define $\binom{a}{b} =0$ whenever $a< b$. 
	\begin{lemma}\label{lem:dim_edge_ideal}	
		If $[\gamma,\gamma']=\tau\in\Delta_1^\circ$ is an edge with vertices $\gamma$ and $\gamma'$ and $\tJ(\tau)=\langle \ell_\tau^{r+1}\rangle\cap \hat\fm_\gamma^{s+1}\cap \hat\fm_{\gamma'}^{s+1}= \bigl\langle \ell_\tau^{s+1-i}\ell_{\tau,\gamma}^i\ell_{\tau,\gamma'}^ i\colon 0\leq i\leq s-r\bigr\rangle$ is the ideal defined in \eqref{eq:edges}, then
		\begin{align}\label{eq:dimE}
			\dim \bigl(\tS/\tJ(\tau)\bigr)_d 
			&= 
			\binom{d+2}{2}- \sum_{i=0}^{s - r } \binom{d-s  -i+1}{2} + \sum_{i=1}^{s  -  r } \binom{d-s  - i}{2}\,\nonumber\\
			&=
			\binom{d+2}{2}+\binom{d-2s+r}{2}- (d-s)^2\,.
		\end{align}
	\end{lemma}
	\begin{proof}
		By a change of coordinates we may assume $\tJ(\tau) = \left\langle x^{s +1-i} y^{i}z^{i}\colon 0 \leq i \leq s - r  \right\rangle\,.$ The following sequence is exact, 
		\begin{equation}\label{eq:resEideal}
			0\rightarrow \bigoplus_{i=1}^{s -r}\tS(-s -i-2)
			\xrightarrow{\varphi_2} 
			\bigoplus_{i=0}^{s - r } \tS(-s -i-1)
			\xrightarrow{\varphi_1} 
			\tS
			\rightarrow \tS/\tJ(\tau)
			\rightarrow 0\ , 
		\end{equation}
		where $\varphi_1$ is defined by $(x^{s+1}, \ x^{s }yz, \dots, x^{ r +1}y^{s - r }z^{s - r })$, and $\varphi_2$ by the  $(s - r +1) \times (s - r )$ matrix
		\begin{equation*}
			\begin{pmatrix}
				yz & 0  & \cdots & 0 \\
				-x & yz & \cdots & 0 \\
				0  & -x & \cdots & 0 \\
				\vdots &\vdots & \ddots &\vdots \\
				0  & 0  & \cdots & yz\\
				0  & 0  & \cdots & -x
			\end{pmatrix}\;.
		\end{equation*} 
		Thus, \[\dim (\tS/\tJ_{\tau})_d = \dim \tS_\tJ -\sum_{i=0}^{s  -  r }\dim \tS(-s  -1-i)_d+\sum_{i=1}^{s  -  r }\tS(-s  -i-2)_d\ ,\]
		which directly leads to Equation \eqref{eq:dimE}.
	\end{proof}
	
	\begin{lemma}\label{lem:dimEborder}	
		If $[\gamma,\gamma']=\tau\in\Delta_1^\circ$, an edge with vertices $\gamma$ and $\gamma'$ and  $\tJ(\tau)=\langle \ell_\tau^{r+1}\rangle\cap \hat\fm_\gamma^{s+1}\cap \hat\fm_{\gamma'}^{r+1}=\langle\ell_\tau^{r+1}\rangle\cap \hat\fm_\gamma^{s+1} = \bigl\langle \ell_\tau^{s+1-i}\ell_{\tau,\gamma}^i\colon 0\leq i\leq s-r\bigr\rangle$ as defined in \eqref{eq:edges}, then 
		\begin{equation}\label{eq:dimEborder}
			\dim \bigl(\tS/\tJ(\tau)\bigr)_d = \binom{d+2}{2}- (s-r+1) \binom{d-s +1}{2} +(s-r)\binom{d-s}{2}\,.
		\end{equation}
	\end{lemma}	
	\begin{proof}
		By a change of coordinates we may assume that $\gamma$ is at the origin of $\R^2$. 
		Using a similar construction to that in Lemma \ref{lem:dim_edge_ideal} leads to resolution of the ideal $\tJ(\tau)$, 
		where $\varphi_1$ and $\varphi_2$ in \eqref{eq:resEideal} are defined by monomials in two variables of degree $s+1$ and $1$ respectively.
	\end{proof} 
	If $\gamma \in \Delta_0$ is a vertex in $\Delta$, we write $\tJ(\gamma)$ for the ideal of $\gamma$ defined in \eqref{eq:ifacesgen}; in our case, $\Delta\subseteq\R^2$ and 
	\begin{equation}\label{eq:vertices}
		{\tJ(\gamma)=
			\sum_{\tau\supset\gamma, \tau\in\Delta_1^\circ}\tJ(\tau)
			=
			\sum_{{\gamma' \in \Delta_0\,,\tau=[\gamma,\gamma']\in\Delta_1}}\langle \ell_\tau^{r_\tau+1}\rangle\cap \hat\fm_\gamma^{s_\gamma+1}\cap \hat\fm_{\gamma'}^{s_{\gamma'}+1}}.
	\end{equation}
	Let us assume the smoothness across the edges $[\gamma, \gamma']=\tau\in\Delta_1^\circ$ is uniform $r_\tau=r$, and there is supersmoothness only at $\gamma$ given by $s_\gamma=s$. Then, $s_\gamma'=r$ in \eqref{eq:vertices}, and we can rewrite it as 
	\begin{equation}\label{eq:vertexone-s}
		\tJ(\gamma)=
		\sum_{\tau\ni\gamma\,,\tau\in\Delta_1}\bigl(\bigl\langle \ell_\tau^{r+1}\rangle\cap \hat\fm_\gamma^{s+1}\bigr).
	\end{equation}
	Before we compute $\dim \tJ(\gamma)_d$ for $\tJ(\gamma)$ in \eqref{eq:vertexone-s}, we prove the following preliminary result. 
	
	\begin{lemma}\label{lem:star-vertex-ideal}
		Let $\fm$ be the maximal ideal in $\tR$ of all polynomials vanishing at $\gamma$, and for each edge $\tau\in\Delta_{1}^\circ$ let $\ell_\tau$ be a linear form vanishing at $\hat\tau$. Then
		\begin{equation*}
			\sum_{\tau_\ni\gamma} \bigl(\hat\fm^{s+1} \cap \langle {\ell}_\tau^{r+1}\rangle\bigr)=\hat\fm^{s+1} \cap \sum_{\tau\ni\gamma}\langle {\ell}_{\tau}^{r+1}\rangle \,.
		\end{equation*}
	\end{lemma}
	\begin{proof}
		We may assume without loss of generality that $\gamma$ is at the origin and take $\fm=\langle x,y \rangle$. 
		Let $f\in \hat\fm^{s+1} \cap\sum_{\tau\ni\gamma}\langle {\ell}_{\tau}^{r+1}\rangle$, then $f=x^iy^jg$ with $i+j=s+1$  $g\in\tS$, and there exist $g_\tau\in \tS$ such that $f=\sum_{\tau\ni\gamma}g_\tau\ell_\tau^{r+1}$.
		
		Notice that $\ell_\tau^{r+1}\in\fm^{r+1}$ for all edges $\tau$ containing $\gamma$. Since $f\in\hat{\fm}^{s+1}$, then we may assume $g_\tau\in \hat\fm^{s-r}$. (Indeed, we may write $f=\sum_{\tau\ni\gamma}g_\tau\ell_\tau^{r+1}=\sum_{\tau\ni\gamma}\bigl(h_\tau\ell_\tau^{r+1}+q_\tau\ell_\tau^{r+1}\bigr)$ with $h_\tau\in\hat{\fm}^{s+1}$ and either $q_\tau=0$ or $q_\tau\notin\hat{\fm}^{s-r}$; $f\in\hat{\fm}^{s+1}$ implies $\sum_{\tau\ni\gamma}q_\tau\ell_\tau^{r+1}=0$.)
		
		In particular, for each $\tau\ni\gamma$ we have $g_\tau\ell_\tau^{r+1}\in \hat\fm^{s+1} \cap \langle {\ell}_\tau^{r+1}\rangle$. 
		It implies $f\in \sum_{\tau_\ni\gamma} \bigl(\hat\fm^{s+1} \cap \langle {\ell}_\tau^{r+1}\rangle\bigr)$.
		The other containment always holds for any choice of ideals. 
	\end{proof}
	\begin{corollary}\label{cor:maximal}
		Let $\Delta\subseteq\R^2$, $\gamma \in \Delta_0^\circ$, and take $\ell_\tau$ and  $\hat\fm$ as in Lemma \ref{lem:star-vertex-ideal}. Then 
		$\hat\fm^{s+1}=\sum_{\tau\in\Delta_1^\circ}\bigl\langle\ell_\tau^{r+1}\bigr\rangle \cap \hat\fm^{s+1}$ for every $s\geq r+\lfloor\frac{r}{t-1}\rfloor$, where $t$ is the number of distinct linear forms $\ell_\tau$.
	\end{corollary}
	\begin{proof}
		Assume, without loss of generality, that $\gamma$ is at the origin of $\R^2$. Thus, the linear forms $\ell_\tau$ are polynomials in two variables in $\fm$.
		By \cite[Theorem 2.6]{geramita1998fat}, the socle degree of the ideal $\bigr\langle\ell_\tau^{r+1}\colon \tau \ni \gamma\bigr\rangle$ is $\displaystyle r+\biggl\lfloor\frac{r}{t-1}\biggr\rfloor$.
		Then, the hypothesis $\displaystyle s\geq r+\frac{r}{t-1}$ implies that $\hat\fm^{s+1}\subseteq\bigl\langle\ell_\tau^{r+1}\colon \tau \ni \gamma\bigr\rangle$, and this in combination with Lemma \ref{lem:star-vertex-ideal}  yields the equality. 
	\end{proof}
	Similarly to the two cases of ideals $\tJ(\tau)$ we considered in Lemmas \ref{lem:dim_edge_ideal} and \ref{lem:dimEborder}, we now compute the dimension of the ideal $\tJ(\gamma)$ in \eqref{eq:vertexone-s}.
	
	\begin{lemma}\label{lemma:dimvertex}
		Let $\Delta\subseteq\R^2$, $0\leq r\leq s\leq d$ be integers, and $\gamma\in\Delta_0^\circ$ a vertex. If $t$ is the number of edges $\tau\in\Delta_1^\circ$ with different slopes containing $\gamma$ and $\tJ(\gamma)=\langle\ell^{r+1}_\tau\colon\tau\ni\gamma,\tau\in\Delta_1^\circ\rangle\cap \hat{\fm}^{s+1}$, then 
		\begin{equation*}
			\dim (\tS/\tJ(\gamma))_d
			=\begin{cases} \binom{d+2}{2}-\frac{t}{2}(d-s)(d+s-2r+1)+b\binom{d+2-\Omega}{2}+a\binom{d+1-\Omega}{2}&\mbox{ if } s<\Omega-1,\\
				\binom{s+2}{2}	&\mbox{otherwise},
			\end{cases}
		\end{equation*}
		where  $\Omega= \lfloor \frac{tr}{t-1}\rfloor +1$, $a= t(r+1)+(1-t)\Omega$, and $b=t-a-1$.
	\end{lemma}
	\begin{proof}
		We can assume without loss of generality that $\gamma$ is at the origin of $\R^2$. Put $\tI(\gamma)=\langle\ell_\tau^{r+1}\colon \tau \ni \gamma\rangle\subseteq \R[x,y]$. 
		Then, $\fm=\langle x, y\rangle$ and $\tI(\gamma)\subseteq \fm$. We can write $\tJ(\gamma)=\bigl(\tI(\gamma)\cap \fm^{s+1}\bigr)\otimes \R[z]$, and so  $\tJ(\gamma)_{d}=(\tI(\gamma)_{\geq s+1}\otimes\R[z])_d$. 
		Hence $\dim (\tS/\tJ(\gamma))_d=\dim \tS_d-\sum_{k=s+1}^d\dim \tI(\gamma)_k=\dim \tS_d-\dim (\tI(\gamma)\otimes\R[z])_d+\dim (\tI(\gamma)\otimes\R[z])_s$.
		By the resolution of the ideal $\tI(\gamma)$ in \cite[Theorem 2.7]{geramita1998fat}, for any $k\geq r$ we have
		\begin{equation}\label{eq:vertexdim}
			\dim (\tS/\tI(\gamma)\otimes\R[z])_k = \binom{k+2}{2}-t\binom{k-r+1}{2} + b\binom{k+2-\Omega}{2}+ a\binom{k-\Omega+1}{2}. 
		\end{equation}
		The statement follows directly by applying \eqref{eq:vertexdim} with $k=d$ and $k=s$. Notice that if $s=k<\Omega-1$ the terms in $b$ and $a$ in \eqref{eq:vertexdim} vanish; 
		if $s\geq \Omega-1$, by  Corollary \ref{cor:maximal} we have $\tJ(\gamma)=\hat\fm^{s+1}$, 
		and  $\dim (\hat\fm^{s+1})_d= \binom{d+2}{2}-\binom{s+2}{2}$.
	\end{proof}	
	\subsection{Supersplines on vertex stars}\label{sec:stars}
	We devote this section to triangulation $\Delta\subseteq \R^2$ which are the star of a vertex i.e., all the triangles $\sigma\in\Delta$ share a common vertex $\gamma$.
	In this case, we write $\Delta=\St(\gamma)$ and say that $\Delta$ is a vertex star, or the star of the vertex $\gamma$.
	
	As before, we take integers  $0\leq r\leq s\leq d$. 
	We write $\calS_d^{r,s}(\Delta^\circ)$ for the set of $C^r$-splines on a vertex star $\Delta=\St(\gamma)$ with supersmoothness $s \geq r$ at the vertex $\gamma$.
	In terms of Definition \ref{def:supersplines}, we have
	$\calS_d^{r,s}(\Delta^\circ)
	= \calS_d^{r,\bm{s}}(\Delta)$ where $\bm{s}$ is given by $s_\gamma=s$, and $s_{\gamma'}=r$ for all $\gamma'\in\partial\Delta$.
	In particular, for each interior edge $[\gamma,\gamma']=\tau\in\Delta_1^\circ$ the ideal $\tJ(\tau)=\langle\ell_{\tau}^{r+1}\rangle\cap \hat\fm_{\gamma}^{s+1}\cap \hat\fm_{\gamma'}^{r+1}=\langle\ell_{\tau}^{r+1}\rangle\cap \hat\fm_{\gamma}^{s+1}$ as in Lemma \ref{lem:dimEborder}.
	\begin{theorem}\label{theo:vetexstardim}
		Let $0\leq r\leq s \leq d$ be integers. If $\Delta=\St(\gamma)$ and $\gamma$ is an interior vertex, then 
		\begin{multline*}
			\dim \calS_d^{r,s}(\Delta^\circ) = \binom{d+2}{2}+f_1^\circ (s-r+1)\binom{d-s+1}{2}-f_1^\circ(s-r)\binom{d-s}{2}  - \dim \tJ(\gamma) 
		\end{multline*}
		where $f_1^\circ$ is the number of interior edges of $\Delta$ and $\dim\tJ(\gamma)$ is given in Lemma \ref{lemma:dimvertex}.
	\end{theorem}
	\begin{proof}
		Put $\tJ(\tau)=\langle\ell_\tau^{r+1}\rangle\cap\hat\fm^{s+1}$ and $\tJ(\gamma)=\sum\limits_{\tau\in\Delta_1^\circ}\tJ(\tau)$.
		Consider the complex 
		\begin{equation}\label{eq:starcomplex}
			0
			\rightarrow\bigoplus_{\sigma\in\Delta} \tS[\sigma]
			\xrightarrow{\overline{\partial}_2}\bigoplus_{\tau\in\Delta_1^\circ} \tS/\tJ(\tau)
			\xrightarrow{\overline{\partial}_1}\tS/\tJ(\gamma)
			\rightarrow 0\,.
		\end{equation}
		Using similar arguments to those in Corollary \ref{cor:kernelcomplexR/J} and Proposition \ref{thm:spline_space_isomorphism}, we get $\dim \calS_d^{r,s}(\Delta^\circ) = \ker(\overline{\partial}_2)$.
		The Euler-Poincar\'e characteristic of the complex \eqref{eq:starcomplex} leads to $\dim  \calS_d^{r,s}(\Delta^\circ) = \binom{d+2}{2} +\dim \sum_{\tau\in\Delta_1^\circ}(\tJ(\tau))_d - \dim (\tJ(\gamma))_d $. 
		The formula in the statement follows by applying Lemma \ref{lem:dim_edge_ideal} and Lemma \ref{lemma:dimvertex} to the previous equality.
	\end{proof}
	\begin{corollary}\label{cor:dimstar_socle}
		Let $\Delta$ be as in Theorem \ref{theo:vetexstardim}, and $t$ be the number of edges with different slopes containing $\gamma$ as a vertex. If  $s\geq r +\lfloor\frac{r}{t-1}\rfloor$, then 
		\begin{equation*}
			\dim \calS_d^{r,s}(\Delta^\circ) = f_1^\circ (s-r+1)\binom{d-s+1}{2}-f_1^\circ(s-r)\binom{d-s}{2} +\binom{s+2}{2}\,. 
		\end{equation*}
	\end{corollary}
	\begin{proof}
		The formula follows by  Theorem \ref{theo:vetexstardim} and the case $s\geq \Omega-1$ in Lemma \ref{lemma:dimvertex}.
	\end{proof}
	\section{A lower bound on the dimension of superspline spaces on triangulations}\label{sec:bounds}
	Throughout this section we assume $\Delta$ is a pure and hereditary simplicial complex in $\R^2$  isomorphic to a disk. 
	
	Since $\dim H_0(\calJ)_d\geq 0$ for any degree $d\geq 0$,  then by Equation \eqref{eq:dimformuladisk} for any choice of smoothness $\br=\{r_\tau\colon \tau\in\Delta_1^\circ\}$ and supersmoothness $\bs=\{s_\gamma\colon \gamma\in\Delta_0\}$ we have   
	\begin{equation}\label{eq:lowerbound}
		\dim \calS_d^{\br,\bs}(\Delta)\geq \binom{d+2}{2} + \sum_{\tau\in\Delta_1^\circ}\dim \tJ(\tau)_d - \sum_{\gamma\in\Delta_0^\circ}\dim \tJ(\gamma)_d\,.
	\end{equation}
	{
		In fact, similarly to the case of splines with global uniform smoothness $r$, it can be shown that the homology module $H_0(\calJ)$ has finite length i.e., $H_0(\calJ)_d=0$ for degree  $d \gg 0$. 
		The proof of this result follows by a straightforward application of the ideas in \cite[Lemma 3.2]{schenck1997local}. 
		\begin{lemma}\label{lemma:homology}
			Let $\Delta\subseteq \R^2$ and $\calJ$ be the complex of ideals  associated to $\br=\{r_\tau\colon \tau\in\Delta_1^\circ\}$ and $\bs=\{s_\gamma\colon \gamma\in\Delta_0\}$. Then, $H_0(\calJ)_d=0$ for all $d\gg 0$. 
		\end{lemma}
		\begin{proof}  
			For $\tau\in\Delta_1^\circ$, by definition \eqref{eq:edgesgen}, the ideal $\tJ(\tau)= \langle \ell_\tau^{r+1}\rangle\cap \hat\fm_\gamma^{s_\gamma+1}\cap \hat\fm_{\gamma'}^{s_{\gamma'}+1}$, where $\hat\fm_{\gamma}$ and $\hat\fm_{\gamma'}$ are the ideals of polynomials in $\tS$ vanishing at $\hat{\gamma}$ and $\hat{\gamma}'$, respectively. 
			If $s=\max\{s_{\gamma'}, s_{\gamma},r_\tau\}$, then in particular  $\ell_{\tau'}^{s+1}\in\tJ(\tau)$. 
			After this observation, one can show that $H_0(\calJ)$ is of finite length following the same steps as in the proof of \cite[Lemma 3.2]{schenck1997local}. (See Appendix \ref{appendix}.)
		\end{proof}	
		Lemma \ref{lemma:homology} and Equation \eqref{eq:dimformuladisk} lead directly to the following theorem.	
		\begin{theorem}\label{thm:lowerbound_full}
			If $\Delta\subseteq \R^2$, and $d \gg 0$, equality holds in \eqref{eq:lowerbound}, i.e.,
			\begin{equation*}
				\dim \calS_d^{\br,\bs}(\Delta) = \binom{d+2}{2} + \sum_{\tau\in\Delta_1^\circ}\dim \tJ(\tau)_d - \sum_{\gamma\in\Delta_0^\circ}\dim \tJ(\gamma)_d\,.
			\end{equation*}
		\end{theorem}
		We now use the results on vertex stars in Section \ref{sec:stars}, and prove a lower bound formula on $\dim \calS_d^{\br,\bs}(\Delta)$ for any $d\geq 0$. 
		First, for a vertex $\gamma\in\Delta_0^\circ$, we 
		define  the ideal 
		\begin{equation}\label{eq:barideal}
			\overline\tJ(\gamma) = \sum_{\tau\ni \gamma}\langle \ell_\tau^{r_\tau+1}\rangle\cap \hat\fm_\gamma^{s_\gamma+1},
		\end{equation}
		where, as before, $\hat\fm_{\gamma}\subseteq \tS$ is the ideal of polynomials vanishing at $\hat{\gamma}$.
		
		The following lemma relates dimension of the ideals $\overline{\tJ}(\gamma)$ and $\tJ(\gamma)$ in degree $d$.
		We show this result following the ideas in the proof of \cite[Lemma 3.2]{schenck1997local}. 
		
		Recall that the \emph{link} of a vertex $\gamma$ in $\Delta\subseteq \R^2$, denoted  $\lk(\gamma)$, is the set of all edges (and their vertices) in $\St(\gamma)$ which do not contain $\gamma$. 
		\begin{lemma}\label{lem:simpler_ideal}
			Let $\Delta\subseteq \R^2$, and $\gamma \in \Delta_0^\circ$. Then,   $\dim \tJ(\gamma)_d \leq \dim \overline{\tJ}(\gamma)_d$ for every $d \geq 0$, equality holds if $d\gg 0$.
		\end{lemma}
		\begin{proof}
			If $\gamma\in\Delta_{0}^\circ$, the ideal $\tJ(\gamma)$ defined in \eqref{eq:ifacesgen} is given by
			\begin{equation}\label{eq:vertexideal}
				\tJ(\gamma)=\sum_{\tau=[\gamma,\nu]}\langle \ell_\tau^{r_\tau+1}\rangle\cap \hat\fm_\gamma^{s_\gamma+1}\cap \hat\fm_{\nu}^{s_{\nu}+1},
			\end{equation}
			where $\hat\fm_{\gamma}$ and  $\hat\fm_{\nu}$ are the ideals in $\tS$ of polynomials vanishing on $\hat\gamma$ and $\hat{\nu}$, respectively, for every vertex $\nu\in\lk(\gamma)$.
			Then, clearly, for any set of non-negative integers $\br= \{r_\tau\colon \tau\in\St(\gamma)_1^\circ\}$ and $\bs= \{s_\gamma\colon \gamma\in\St(\gamma)_0^\circ\}$ we have $\tJ(\gamma)\subseteq \overline\tJ(\gamma)$ proving the first claim.
			
			The second claim can be proved by showing that for a large enough $N$ and $\langle x,y,z\rangle^N$ annihilates $\overline{\tJ}(\gamma)/\tJ(\gamma)$.
			Let $s = \max\bigl\{s_{\nu}, r_\tau\colon \nu,\tau\in\St(\gamma)\bigr\}$ and let $\tau' = [\gamma,\gamma']$ and $\tau'' = [\gamma,\gamma'']$ be two  edges with distinct slope that contain the vertex $\gamma$.
			Take $p=\prod_{\tau\in\lk(\gamma)}\ell_{\tau}$, where $\ell_\tau$ denotes a choice of a linear form in $\tS$ vanishing on $\hat{\tau}$.
			Since $\ell_{\tau'} \in \hat{\fm}_\gamma \cap \hat{\fm}_{\gamma'}$,  $\ell_{\tau''} \in \hat{\fm}_\gamma \cap \hat{\fm}_{\gamma''}$, and $p \in \hat{\fm}_{\nu}$ for any $\nu\in\lk(\gamma)$, then for any $f \in \overline{\tJ}(\gamma)$ we have
			\begin{equation*}
				\ell_{\tau'}^{s+1}f,\;\ell_{\tau''}^{s+1}f,\;p^{s+1}f \in \tJ(\gamma)\,.
			\end{equation*}
			But $\langle x,y,z\rangle^N\subseteq \langle \ell_{\tau'},\ell_{\tau''},p \rangle$ for some $N\gg 0$, and the claim follows.
		\end{proof}	
		\begin{theorem}\label{thm:lowerbound_alternate}
			Let $\Delta\subseteq \R^2$ be a simplicial complex homeomerphic to a disk, then 
			\begin{equation*}
				\dim \calS_d^{\br,\bs}(\Delta) \geq \binom{d+2}{2} + \sum_{\tau\in\Delta_1^\circ}\dim \tJ(\tau)_d - \sum_{\gamma\in\Delta_0^\circ}\dim \overline{\tJ}(\gamma)_d\,, 
			\end{equation*}
			for every $d \geq 0$, and equality holds if $d \gg 0$. 
		\end{theorem}
		\begin{proof}
			The claim follows from Theorem \ref{thm:lowerbound_full} and Lemma \ref{lem:simpler_ideal}.
		\end{proof}	
	}
	In the following, we consider the space of splines $\calS_d^{r,s}(\Delta)$ with fixed global smoothness $r$ and uniform supersmoothness $s$ at all the vertices of $\Delta$.
	As a corollary of Theorem \ref{thm:lowerbound_alternate} prove a combinatorial lower bound formula on the dimension of $\calS_d^{r,s}(\Delta)$.
	\begin{corollary}\label{cor:lowerbound}
		If $\Delta\subseteq \R^2$ and $\calS_d^{r,s}(\Delta)$ is the set of splines with uniform smoothness $r$ across all edges $\tau\in\Delta_1^\circ$ and supersmoothness $s$ at all the vertices $\gamma\in\Delta_0$, then 
		\begin{align}\label{eq:lowerbound_computable}
			\dim \calS_d^{r,s}(\Delta)&\geq 	
			\binom{d+2}{2}+\biggl[(d-s)^2 -\binom{d-2s+r}{2} \biggr]f_1^\circ
			- \sum_{\gamma\in\Delta_0^\circ}\dim \overline{\tJ}(\gamma)_d\,,
		\end{align}	
		where \begin{equation*}
			\dim \overline{\tJ}(\gamma)_d
			=\begin{cases} \frac{t_\gamma}{2}(d-s)(d+s-2r+1)-b_\gamma\binom{d+2-\Omega_\gamma}{2}-a_\gamma\binom{d+1-\Omega_\gamma}{2}&\mbox{ if } s<\Omega_\gamma-1,\\
				\binom{d+2}{2}-\binom{s+2}{2}	&\mbox{otherwise},
			\end{cases}
		\end{equation*}
		with $t_\gamma$ the number of edges with different slopes containing the vertex $\gamma$,  $\Omega_\gamma= \lfloor \frac{t_\gamma r}{t_\gamma-1}\rfloor +1$, $a_\gamma= t_\gamma(r+1)+(1-t_\gamma)\Omega$, and $b_\gamma=t_\gamma-a_\gamma-1$.
	\end{corollary}
	\begin{proof}
		The formula follows from Theorem \ref{thm:lowerbound_alternate} together with the formulas for $\dim \tJ(\tau)_d$  and  $\dim \overline{\tJ}(\gamma)_d$ in Lemmas \ref{lem:dim_edge_ideal} and  \ref{lemma:dimvertex}, respectively. 
	\end{proof}	
	
	In the examples in Section \ref{sec:examples} we compare the lower bounds \eqref{eq:lowerbound} and \eqref{eq:lowerbound_computable} for specific triangulations; we also consider the homology modules $H_0(\calJ)$ and give their explicit description. 
	
	We briefly comment that an upper bound can be proved on $\dim \calS_d^{\br,\bs}(\Delta)$ following a similar argument to that used in the case of splines $\calS_d^r(\Delta)$ with global uniform smoothness $r$ by Mourrain and Villamizar in \cite{mourrain2013homological}. 
	Namely, we fix a numbering  $\gamma_1,\dots,\gamma_{f_0^\circ}$ on the interior vertices of $\Delta$. 
	For each vertex $\gamma_i$, denote by $N(\gamma_i)$ the set of edges that connect $\gamma_i$ to any of the first $i-1$ vertices in the list or to a vertex on the boundary, and define the ideal $\tilde{\tJ}(\gamma_i)=\sum_{\tau\in N(\gamma_i)}\tJ(\tau)$.
	\begin{proposition}
		The dimension of $\calS^{\br,\bs}_d(\Delta)$ is bounded above by 
		\begin{equation}\label{eq:upperbound}
			\dim \calS_d^{\br,\bs}(\Delta)\leq \binom{d+2}{2} + \sum_{\tau\in\Delta_1^\circ}\dim \tJ(\tau)_d - \sum_{i=1}^{f_0^\circ}\dim \tilde\tJ(\gamma_i)_d\,.
		\end{equation}
	\end{proposition}
	\begin{proof}
		The argument used in \cite[Theorem 2]{mourrain2013homological} is independent of the ideals $\tJ(\tau)$ associated to the edges $\tau\in\Delta_1^\circ$, and therefore it immediately leads to the upper bound in Equation \eqref{eq:upperbound}.
	\end{proof}
	Similarly as in the case of $\dim \calS_d^r(\Delta)$, an explicit upper bound formula requires the computation of $\dim\tJ(\gamma)$, and the following result follows immediately by comparing the lower and the upper bound in \eqref{eq:lowerbound} and  \eqref{eq:upperbound}, respectively.
	\begin{corollary}\label{cor:hom0}
		If $\Delta\subseteq\R^2$ is a simplicial complex homeomorphic to a disk such that $\tJ(\gamma)=\tilde\tJ(\gamma)$ for all $\gamma\in\Delta_0^\circ$ then equality holds in \eqref{eq:lowerbound}.
	\end{corollary}
	
	\begin{example}[Optimality of lower bounds]\label{ex:delaunay}
		We generate a random triangulation $\Delta$, shown in Figure \ref{fig:delaunay}, for $r=2$ we consider the space $\calS_d^{r,\bs}(\Delta)$ of $C^r$-continuous splines on $\Delta$ with supersmoothness $\bs=\{s_\gamma\colon\gamma\in\Delta_0^\circ\}$ with $s_\gamma \in \{2, 3, 4\}$. 
		We compare the lower bound in Corollary \ref{cor:lowerbound} \eqref{eq:lowerbound_computable} 
		with the exact dimension of $\calS_d^{r,\bs}(\Delta)$, which is a subspace of $\calS_d^2(\Delta)$.
		In particular, we randomly assign supersmoothness $s_\gamma \in \{2, 3, 4\}$ to vertices $\gamma\in \Delta_0$.
		With reference to Figure \ref{fig:delaunay}, the vertices encircled once correspond to $s_\gamma = 3$, the ones encircled twice correspond to $s_\gamma = 4$, and the others correspond to $s_\gamma = 2$.
		As shown in Table \ref{tab:delaunay}, the explicit bound from Corollary \ref{cor:lowerbound} coincides with the lower bound in Equation \eqref{eq:lowerbound} as well as the dimension of $\calS_d^{r,\bs}(\Delta)$ in large degree. 
		In fact, in this case the equality between the dimension of the vertex ideals \eqref{eq:barideal} and \eqref{eq:vertexideal} in Lemma \ref{lemma:dimvertex} holds for every  $d\geq 4$. 
	\end{example}
	
	\begin{figure}[ht]
		\centering
		\includegraphics[trim=.4cm 1.4cm 0.7cm 0.3cm,clip,scale=1.0]{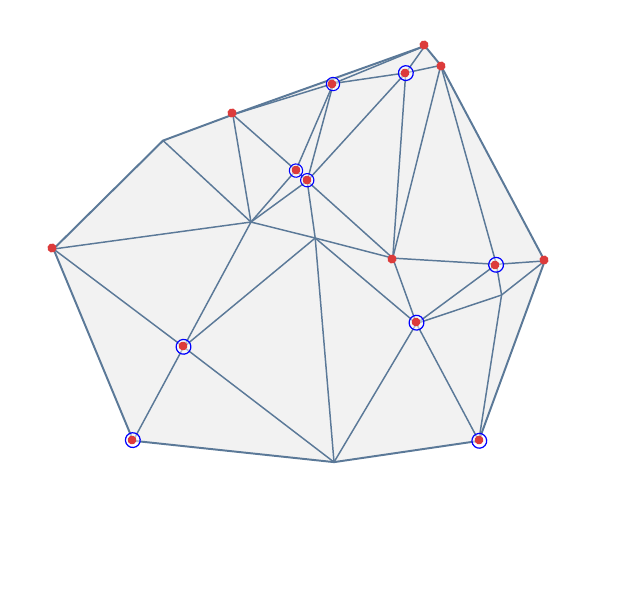}
		\caption{A randomly generated triangulation, the smoothness across all edges is $r = 2$, and additional smoothness $s - r \in \{0,1,2\}$ is assigned randomly to all vertices.
			Above, $s_\gamma-r = 1$ for vertices encircled once, $s_\gamma-r=2$ for vertices encircled twice, and $s_\gamma-r  = 0$ otherwise.
			The exact dimensions of $\calS_d^{r,\bs}$ and the lower bounds $\mathrm{LB}\eqref{eq:lowerbound}$ and $\mathrm{LB}\eqref{eq:lowerbound_computable}$ are given in Table \ref{tab:delaunay} for different choices of $d$.
		}\label{fig:delaunay}
	\end{figure}
	\renewcommand{\arraystretch}{1.0}
	\begin{table}[htbp]
		\centering
		\begin{tabular}{|>{\centering\arraybackslash}m{1cm}|>{\centering\arraybackslash}m{2.5cm}|>{\centering\arraybackslash}m{2cm}|>{\centering\arraybackslash}m{2cm}|>{\centering\arraybackslash}m{2.5cm}|}
			\hline\hline
			\rowcolor{gray!10}
			\thead[c]{$d$} & \thead[c]{$\dim H_0(\calJ)_d$} & \thead[c]{$\mathrm{LB}\eqref{eq:lowerbound_computable}$} & \thead[c]{$\mathrm{LB}\eqref{eq:lowerbound}$} &
			\thead[c]{$\dim \calS_d^{2,\bs}(\Delta)$} \\
			\hline\hline
			4 & 4 & 15 & 15 & 15 \\\hline
			5 & 0 & 30 & 31 & 31 \\\hline
			6 & 0 & 108 & 108 & 108 \\\hline
			7 & 0 & 223 & 223 & 223 \\
			\hline
			\hline
		\end{tabular}
		\caption{Lower bounds and dimension for the superspline space $\calS_d^{2,\bs}(\Delta)$ in Example \ref{ex:delaunay}, where $\Delta$ is the triangulation shown in Figure \ref{fig:delaunay}.
			Here, $\mathrm{LB}\eqref{eq:lowerbound}$ and $\mathrm{LB}\eqref{eq:lowerbound_computable}$ are the lower bounds from \eqref{eq:lowerbound} and \eqref{eq:lowerbound_computable}, respectively. 
		}\label{tab:delaunay}
	\end{table}
	\section{Examples}\label{sec:examples}
	\subsection{Argyris superspline space}\label{subsec:argyris_dimension_general}
	Let $\Delta\subseteq \R^2$ be a triangulation homeomorphic to a disk, and $r\geq 0$ an integer. In this example we compute the dimension of the  superspline space $\calS_{4r+1}^{r,2r}(\Delta)$. 
	A particular case, taking $r=1$ is called the Argyris element $\calS_{5}^{1,2}(\Delta)$, which was introduced in the finite-element literature in \cite{zlamal68,zen70}. 
	A description of the Argyris space, and the general case $\calS_{4r+1}^{r,2r}(\Delta)$ using Bernstein--B\'ezier techniques is included in \cite[Chapter 6--8]{lai2007spline}.
	
	Following Definition \ref{def:supersplines}, the space $\calS^{r,2r}_{4r+1}(\Delta)$ corresponds to the set
	\begin{equation*}
		\calS^{r,2r}_{4r+1}(\Delta) = 
		\bigl\{f\in \calS^{r}_{4r+1}(\Delta) \colon f\in C^{2r}(\gamma) \text{\ for all \ } \gamma \in \Delta_0\bigr\}\,.
	\end{equation*}
	If $\gamma\in\Delta_0^\circ$ is an interior vertex, then there are at lest three edges having $\gamma$ as one of their vertices, and at least two of them, say $\tau$ and $\tau'$, have different slopes.
	Let $\ell$ be a linear form vanishing on the plane containing $\hat \gamma$ and $\hat\gamma'$. After a suitable change of coordinates we can write
	\begin{align*}
		\tJ(\gamma) \supseteq \tJ({\tau}) + \tJ({\tau'}) 
		&=\bigl\langle \ell_\tau^{2r+1-i}\ell_{\tau'}^i\ell^i\ ,\  \ell_{\tau'}^{2r+1-i}\ell_\tau^i\ell^i\colon 0 \leq i \leq r  \bigr\rangle \\
		&= \langle x^{2r+1-i}y^i z^i\ ,\  x^i y^{2r+1-i} z^i\colon 0 \leq i \leq r  \rangle \,.
	\end{align*}
	Then, every monomial $x^a y^b z^c\in\tS$, with $a+b+c=4r+1$ and $0 \leq c \leq 2r$, is contained in $\tJ(\gamma)_{4r+1}$. Thus 
	\begin{equation}\label{eq:isoargyris}
		\tJ(\gamma)_{4r+1}
		\cong \tS_{4r+1} / \langle z^{2r+1} \rangle_{4r+1},
	\end{equation}
	and  $\dim \tJ(\gamma)_{4r+1} = \binom{4r+3}{2} - \binom{2r+2}{2} = (2r+1)(3r+2)$.
	
	By \cite[Lemma 3.3]{schenck1997local} (see Lemma \ref{lemma:ordering}), we know that there exists a numbering of the vertices of $\Delta$ such that every interior vertex $\gamma\in\Delta_0^\circ$ is connected to two vertices with smaller index by edges which have distinct slopes.
	Taking such an ordering on the vertices of $\Delta$, if $\gamma\in\Delta_0^\circ$, denote by $\tilde\tJ(\gamma)$ the sum of ideals $\tJ(\tau)$ associated to the edges $\tau$ containing $\gamma$ and whose other vertex is of smaller index than $\gamma$.
	Since the number of those edges with different slope is at least two, then  $\tilde\tJ(\gamma)_{4r+1}=\tJ(\gamma)_{4r+1}$.
	Thus, Corollary \ref{cor:hom0} implies $\dim H_0(\calJ)_{4r+1} =0$\,.
	
	On the other hand, for any edge $\tau\in\Delta_1^\circ$, the edge ideal $\tJ$ can be written as $\tJ(\tau) = \langle x^{2r+1-
		i}y^iz^i\colon 0 \leq i \leq r  \rangle\,.$
	Then $\dim \tJ(\tau)_{4r+1}$ is given in Lemma \ref{lem:dim_edge_ideal}. 
	Applying the dimension formula \eqref{eq:dimformuladisk}, together to \eqref{eq:dimE} and\eqref{eq:isoargyris}, we get
	\begin{align}\label{eq:argyris_general}
		\dim \calS^{r,2r}_{4r+1}(\Delta) 
		&=\binom{4r+3}{2} + f_1^\circ (2r+1)^2-f_1^\circ\binom{r+1}{2} 
		-f_0^\circ(2r+1)(3r+2)\\
		&=\binom{2r+2}{2}f_0+\binom{r+1}{2}f_1+\binom{r}{2}f_2\, .\nonumber
	\end{align}
	The last equality follows by the Euler relation $3f_2 = f_1 + f_1^\circ$. A proof of \eqref{eq:argyris_general} using Bernstein--B\'ezier methods is in 
	\cite[Theorem 8.1]{lai2007spline}.
	
	\subsection{Intrinsic supersmoothness and degenerate spaces on vertex stars} 
	Let us consider $\Delta=\St(\gamma)\subseteq\R^2$ be the star of the vertex $\gamma$. 
	For any pair of integers $0\leq r\leq d$ we have 
	\begin{equation}\label{eq:usualdimformulavertexstar}
		\dim \calS_d^r(\Delta)= \binom{d+2}{2} +(f_1^\circ-t)\binom{d-r+1}{2} + b\binom{d+2-\Omega}{2}+ a\binom{d-\Omega+1}{2}, 
	\end{equation}	
	where $t$ is the number of different slopes of the edges containing $\gamma$, $\Omega= \lfloor \frac{tr}{t-1}\rfloor +1$, $a= t(r+1)+(1-t)\Omega$, and $b=t-a-1$.
	
	The dimension formula \eqref{eq:usualdimformulavertexstar} was proved by Schumaker \cite{schumaker1984bounds}. The notation we use here follows the algebraic approach to proof this formula by Schenck and Stillman in \cite{schenck1997family} and Mourrain and Villamizar in \cite{mourrain2013homological}. 
	
	Notice that for any $s\geq r$, we have $\calS_d^{r,s}(\Delta)\subseteq\calS_d^{r,s}(\Delta^\circ)\subseteq\calS_d^r(\Delta)$, where as before, $\calS_d^{r,s}(\Delta^\circ)$ is the set of $C^r$-splines on $\Delta$ with supersmoothness $s$ at $\gamma$.
	It is clear that the set $\calS_d^{r,s}(\Delta)$ contains all the trivial splines, also called global polynomials, on $\Delta$ i.e., the splines $F$ on $\Delta$ whose restriction $F|_\sigma=f$ to each face $\sigma\in\Delta$ is the same polynomial $f\in \tR$. 
	Therefore if $\dim \calS_d^r(\Delta)=\binom{d+2}{2}$ then both $\calS_d^{r,s}(\Delta)$ and $\calS_d^{r,s}(\Delta^\circ)$ only contain trivial splines. 
	From \eqref{eq:usualdimformulavertexstar} it is easy to see that $\dim \calS_d^r(\Delta)=\binom{d+2}{2}$ for all $d\leq \Omega$ when $f_1^\circ > t$, and for all $d\leq r$ in the generic case.
	
	The dimension formula for supersplines spaces proved in Section \ref{sec:topology} can be used to identify unexpected (also called intrinsic) supersmoothness in spaces of $C^r$-splines.
	For example, by computing the exact dimension of the spaces we can provide a short alternative proof of the result by Sorokina in \cite[Theorem 3.1]{sorokina_2010}. Namely, we will show that the $C^r$-splines on any generic vertex star all possess supesmoothness $\lfloor\frac{r+1}{t-1}\rfloor+r$ at the interior vertex. 
	
	Suppose $f_1^\circ=t$, and take $s= \lfloor \frac{r+1}{t-1}\rfloor+r$. 
	Following the notation in Equation \eqref{eq:usualdimformulavertexstar}, we have that
	$s=\Omega$ if $\frac{r+1}{t-1}\in\Z$, and $s=\Omega-1$ otherwise. 
	By Corollary \ref{cor:dimstar_socle} we get
	\begin{equation}\label{eq:integer}
		\dim \calS_d^{r,s}(\Delta^\circ) =
		\begin{cases} t (\Omega-r+1)\binom{d-\Omega+1}{2}-t(\Omega-r)\binom{d-\Omega}{2} +\binom{\Omega+2}{2}\,; &\mbox{ if }  \frac{r+1}{t-1}\in\Z\\[5pt]
			t (\Omega-r)\binom{d-\Omega+2}{2}-t(\Omega-1-r)\binom{d-\Omega+1}{2} +\binom{\Omega+1}{2}\,; &\mbox{ otherwise}\,.	
		\end{cases}
	\end{equation}
	On the other hand, if $\frac{r+1}{t-1}\in\Z$ we have $a=t-1$, $b=0$, and Equation \eqref{eq:usualdimformulavertexstar} leads to
	\begin{equation}\label{eq:noninteger}
		\dim \calS_d^r(\Delta)=
		\begin{cases} 
			\binom{d+2}{2} +  (t-1)\binom{d-\Omega+1}{2}\,; &\mbox{ if }  \frac{r+1}{t-1}\in\Z\\[5pt]
			\binom{d+2}{2} + (t-a-1)\binom{d+2-\Omega}{2}+ a\binom{d-\Omega+1}{2}\,; &\mbox{ otherwise}\,.	
		\end{cases}
	\end{equation}
	A straightforward computation shows that $\dim \calS_d^r(\Delta)=\calS_d^{r,s}(\Delta^\circ)$ in both cases \eqref{eq:integer} and \eqref{eq:noninteger}. 
	
	Similarly, we can show that for vertex stars $\calS^{r,s}(\Delta^\circ)=\calS^r(\Delta)$ if and only if $\dim \calS_s^r(\Delta)= \binom{s+2}{2}$.
	This criterion was proved by Floater and Hu in \cite[Theorem 1]{floater_hu_2020}; they call \emph{degenerated} the spline spaces that only contain trivial splines. 
	
	If we assume that $\calS^r(\Delta)\subseteq \calS^{r,s}(\Delta^\circ)$, by Theorem \ref{theo:vetexstardim}  we know that $\calS_s^{r,s}(\Delta^\circ)=\binom{s+2}{2}$. Then  $\calS_s^r(\Delta)$ is also degenerated. Conversely, if $\calS_s^r(\Delta)=\binom{s+2}{2}$ for $0\leq r<s$, then by \eqref{eq:usualdimformulavertexstar} we have
	\begin{equation*}
		\dim \calS_s^r(\Delta)= \binom{s+2}{2} +(f_1^\circ-t)\binom{s-r+1}{2} + b\binom{s+2-\Omega}{2}+ a\binom{s-\Omega+1}{2}\,, 
	\end{equation*}	
	and this implies that the triangulation is generic i.e., $f_1^\circ = t$, and that $s+2-\Omega\leq 1$, or $s+1-\Omega\leq 1$ and $b=0$. 
	Suppose $s+2-\Omega\leq 1$.  It follows $\Omega \geq s+1$, which is equivalent to say that the generators of the module of syzygies of the forms $\{\ell_\tau^{r+1}\colon \tau\in\Delta_1^\circ\}$ have degree strictly greater than $s-(r+1)$. 
	If we assume $\gamma$ is at the origin then the linear forms $\ell_{\tau}\in\tS$, and therefore the generators of their module of syzygies, only involve the variables $x,y$. 
	As a graded module over $\tS=\R[x,y,z]$, the set $\calS^r(\Delta)$ is generated by trivial splines and splines of the form $G=(g_1\ell_1^{r+1},g_1\ell_1^{r+1}+ g_2\ell_2^{r+1}, \dots, g_1\ell_1^{r+1}+\cdots+g_t\ell_t^{r+1})$, where $g_1\ell_1^{r+1}+\cdots+g_t\ell_t^{r+1}=0$ is a syzygy of the forms $\{\ell_\tau^{r+1}\colon \tau\in\Delta_1^\circ\}$. 
	(An introduction to splines as modules over a ring can be found in \cite[Chapter 8]{UsAlg}.)
	Since all the polynomials $g_i$ are homogeneous in $x,y$ of degree greater or equal to $s+1-(r+1)$, then each polynomial (piece) $g_i\ell_i^{r+1}$ is zero up to order $s$ at $\gamma$. Hence $G\in\calS^s(\gamma)$, which implies that every spline in $C^r(\Delta)$ is in $C^s(\gamma)$.
	
	If $b=0$, then the smallest degree of a syzygy is $\Omega+1$. The condition $\Omega\geq s$  implies $\deg(g_i)\geq s+1 -(r+1)$, hence also in this case $\calS^r(\Delta)\subseteq C^s(\gamma)$ and it follows that $\calS^{r}(\Delta)\subseteq \calS^{r,s}(\Delta^\circ)$.
	Therefore $\calS^{r}(\Delta)\subseteq \calS^{r,s}(\Delta^\circ)$ if and only if $\calS_s^r(\Delta)$ contains only trivial splines.
	In particular, this criterion combined with the result by Sorokina \cite[Theorem 3.1]{sorokina_2010} implies that $s= \lfloor\frac{r+1}{t-1}\rfloor+r$ is the largest order of supersmoothness such that $\calS^r(\Delta)\subseteq \calS^{s,r}(\Delta^\circ)$.
	
	\subsection{Supersmooth splines on Powell--Sabin 6-split refinements}\label{ex:hendrik}
	Let $\Delta\subseteq\R^2$ be a triangulation, and let $\Delta^\star$ be a triangulation obtained from $\Delta$ via a Powell--Sabin six split.
	Namely, we choose a point $Z_\sigma$ in the interior of each triangle $\sigma\in\Delta$ so that if two triangles $\sigma,\sigma'\in\Delta$ share a common edge $\tau=\sigma\cap\sigma'$, then the line joining $Z_\sigma$ and $Z_{\sigma'}$ intersects $\tau$ at a point $B_\tau$ that lies at the interior of $\tau$.
	If $\tau\in\Delta_1$ is an edge on the boundary, we choose an interior point on $\tau$ and denote it by $B_\tau$.
	The set of vertices  $\Delta_0$ of $\Delta$ together with the points $Z_\sigma$ and $B_\tau$, for all $\sigma\in\Delta_2$ and $\tau\in\Delta_1$, are the vertices of the new triangulation $\Delta^\star$. 
	If  $\sigma\in\Delta_2$ is a triangle of $\Delta$, we join $Z_\sigma$ to each vertex of $\sigma$, and to each vertex $B_\tau$ on the edges $\tau\in\sigma$.
	Thus, the Powell--Sabin triangulation $\Delta^\star$ is a refinement of $\Delta$, where each triangle in $\Delta$ has been subdivided into six smaller triangles. 
	An example of a partition along with its Powell--Sabin 6-split is in Figure \ref{fig:MSsplit}.
	
	\begin{figure}[htbp]
		\centering
		\includegraphics[height=3.5cm]{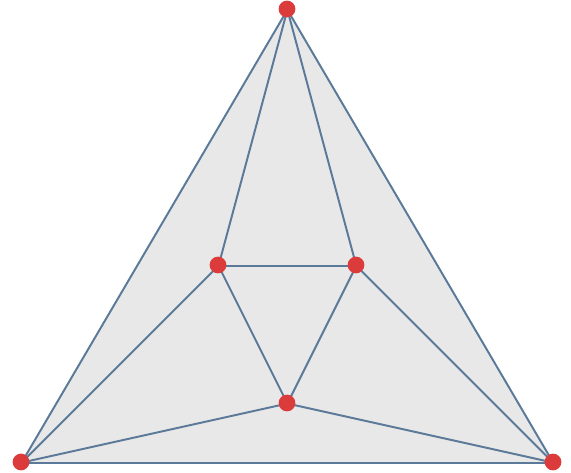}\qquad	\includegraphics[height=3.5cm]{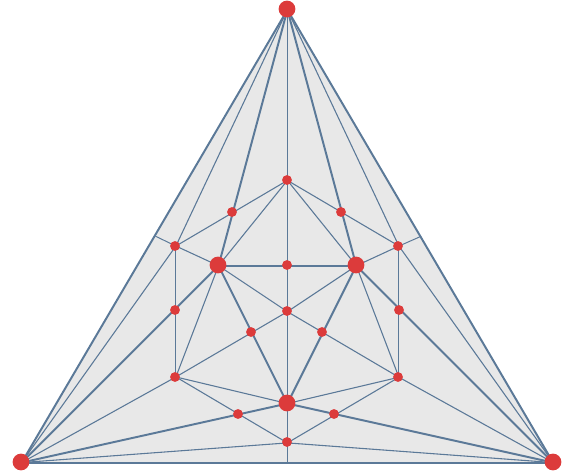}\qquad
		\includegraphics[height=3.5cm]{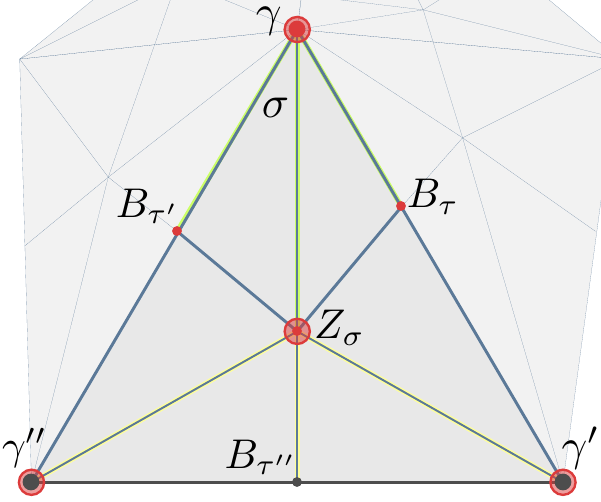}
		\caption{Symmetric Morgan--Scott triangulation (left), and the corresponding Powell--Sabin 6-split applied to each triangle of this triangulation (center). The notation in the 6-split (right) is used in Example \ref{ex:hendrik}; in this case, the vertices $\gamma'$ and $\gamma''$ are assumed to be on the boundary. The smoothness across the edges $[Z_\sigma,B_\tau]$ and $[Z_\sigma,B_{\tau'}]$ and at the vertices $\gamma$, $\gamma'$ and $\gamma''$ is $s\geq r$.  }\label{fig:MSsplit}
	\end{figure}
	In the following, given integers $r \geq 0$, $s \geq \max\{r,2r-1\}$ and $d \geq 2s-r+1$, we compute $\dim \calS_d^{\br,\bs}(\Delta^\star)$, where  
	$\br=\bigl\{r_\tau\colon \tau\in(\Delta^\star)^\circ_1\bigr\}$ and $\bs=\bigl\{s_\gamma\colon \gamma\in(\Delta^\star)_0\bigr\}$ are defined by
	\begin{align*}
		r_\tau&=
		\begin{cases}
			s &\mbox{ if } \tau=[Z_\sigma,B_\beta] \mbox{ for some }	\sigma\in\Delta_2, \mbox{and } \sigma\supseteq	\beta\in\Delta_1\,,\\
			r &\mbox{ otherwise,} 	
		\end{cases}\\
		s_\gamma&=
		\begin{cases}
			s &\mbox{ if } \gamma\in \Delta_0\cup \{Z_\sigma\colon\sigma\in\Delta_2\}\,,\\
			r &\mbox{ if  } \gamma\in\{B_\tau\colon\tau\in\Delta_1^\circ\}\;.
		\end{cases}
	\end{align*}
	The specific choice $r\geq 0$, $s = 2r-1$ and $d=3r-1$ is studied by Speleers in \cite{speleers_2013} using Bernstein--B\'ezier methods.
	
	In our settings, from the dimension formula in Equation \eqref{eq:dimformuladisk} we get
	\begin{equation}\label{eq:PS}
		\dim \calS_d^{\br,\bs}(\Delta^\star)= \binom{d+2}{2} + \sum_{\tau\in(\Delta^\star)_1^\circ}\dim \tJ(\tau)_d - \sum_{\gamma\in(\Delta^\star)_0^\circ}\dim \tJ(\gamma)_d  +\dim H_0(\calJ)_d\,,  
	\end{equation}
	where the ideal $\tJ(\tau)$, for each $\tau\in(\Delta^\star)_1^\circ$,  is defined by
	\begin{equation}\label{eq:casesPSedges}
		\tJ(\tau)= \begin{cases}
			\langle \ell_{\tau}^{s+1}\rangle &\mbox{ if }  \tau=[Z_\sigma,B_\beta] \mbox{ for }	\sigma\in\Delta_2, \mbox{and } \beta\in\Delta_1 ;\\
			\langle \ell_{\tau}^{r+1}\rangle\cap\hat\fm_{\gamma}^{s+1}\cap\hat\fm_{Z_\sigma}^{s+1} &\mbox{ if }  \tau=[Z_\sigma,\gamma] \mbox{ for }	\sigma\in\Delta_2, \mbox{and } \gamma\in\Delta_0 ;\\
			\langle \ell_{\tau}^{r+1}\rangle\cap\hat \fm_{\gamma}^{s+1}&\mbox{ if }  \tau=[B_\beta,\gamma] \mbox{ for }	\beta\in\Delta_1^\circ, \mbox{and } \gamma\in\Delta_0\,, 
		\end{cases} 
	\end{equation}
	and  $\tJ(\gamma)=\sum_{\tau\in\Delta_1^\circ, \gamma\in\tau}\tJ(\tau)$, for each vertex $\gamma\in(\Delta^\star)_0^\circ$\,. Here, as before, if $\tau\in\Delta^\star_1$ and $\gamma\in\Delta^\star_0$, then $\ell_\tau$ is a linear form vanishing on $\hat\tau$, and  $\hat\fm_\gamma$ is the ideal of all polynomials in $\tS$ vanishing at $\hat\gamma$.
	
	Notice that for the ideals $\tJ(\tau)$ in \eqref{eq:casesPSedges}, we have $\dim \tJ(\tau)=\binom{d-s+1}{2}$ if $\tau=[Z_\sigma,B_\beta]$, and  
	$\dim \tJ(\tau)$ in the other two cases follows directly from Equations \eqref{eq:dimE} and \eqref{eq:dimEborder}, respectively. 
	
	The dimension of the ideal $\tJ(\gamma)$ associated to the vertices can be computed as follows. We consider the three types of vertices separately.
	Thereafter, we show that $H_0(\calJ)_d = 0$.
	\begin{description}[style=unboxed,leftmargin=0cm]
		\item[Case 1]
		We show that $\dim \tJ(Z_\sigma)_d = \binom{d+2}{2}-\binom{s+2}{2}$.
		By construction, $\tJ(Z_\sigma)$ is the sum of three ideals of the form $\langle\ell_{\tau}^{s+1}\colon \tau=[Z_\sigma,B_\beta]\rangle\subseteq \hat\fm_{Z_\sigma}^{s+1}$ where $B_\beta$ is the vertex on the edge $\beta\subseteq\sigma$, and three ideals of the form $\langle \ell_{\tau}^{r+1}\rangle\cap \hat\fm_{Z_{\sigma}}^{s+1} \cap\hat\fm_{\nu}^{s+1}$ for the edges $\tau=[Z_\sigma,\nu]$ for vertices $\nu\in\sigma$, $\nu\in\Delta_0$.
		Then, in particular $\tJ(Z_\sigma) \subseteq \hat\fm_{Z_{\sigma}}^{s+1}$. 	
		We want to show that $ x^{i}y^jz^k\in \tJ(Z_\sigma)$ for all monomials of degree $d=i+j+k$ for $d\geq 2s-r+1$,
		such that $i+j=s+1$. 
		
		By a change of coordinates, we may assume $\ell_{[Z_\sigma,\gamma]}=x$,\,  $\ell_{[Z_\sigma,\gamma']}=y$, \, 
		and $\hat\fm_\gamma=\langle x,z\rangle$. 
		Then, $\hat\fm_{Z_{\sigma}}=\langle x,y\rangle$ and $\hat\fm_{\gamma'}=\langle y,z\rangle$.
		
		Since
		\[\bigl\langle \ell_{[Z_\sigma,\gamma]}^{r+1}\bigr\rangle\cap \hat\fm_{Z_{\sigma}}^{s+1} \cap\hat\fm_{\gamma}^{s+1} + \bigl\langle \ell_{[Z_\sigma,\gamma']}^{r+1}\bigr\rangle\cap \hat\fm_{Z_{\sigma}}^{s+1} \cap\hat\fm_{\gamma'}^{s+1} \subseteq \tJ(Z_\sigma)\,,
		\]
		then
		$x^{s+1-i}y^iz^i$ and $y^{s+1-i}x^iz^i$ are elements in $\tJ(Z_\sigma)$, for all $i=0,\dots,s-r$.
		Thus, if $s\geq 2r-1$ this implies that $x^iy^jz^k\in \tJ(Z_\sigma)$ for all $i+j=s+1$ in degree $d\geq 2s-r+1$, except for $x^ry^rz^k$ when $s=2r-1$.
		But in the latter case, since $\ell_{[Z_\sigma,B_{\tau''}]}^{s+1}\in \tJ(Z_\sigma)$ and
		$\ell_{[Z_\sigma,B_{\tau''}]}\in \hat\fm_{Z_{\sigma}}=\langle x,y \rangle$, then it follows $x^ry^rz^k\in\tJ(Z_\gamma)$.
		Consequently, $(\hat\fm_{Z_{\sigma}}^{s+1})_d \subseteq \tJ(Z_\sigma)_d$ and the dimension formula follows.
		
		\smallskip
		
		\item[Case 2]
		Let $\gamma \in \Delta_0^\circ$.
		Similarly as in Case 1, we have $\dim \tJ(Z_\sigma)_d = \binom{d+2}{2}-\binom{s+2}{2}$.
		Indeed, the ideal $\tJ(\gamma)$ is the sum of at least two ideals of the form $\langle\ell_\tau^{r+1}\rangle\cap \hat{\fm}_\gamma^{s+1}$ and three of the form $\langle\ell_\tau^{r+1}\rangle\cap \hat{\fm}_\gamma^{s+1}\cap \hat{\fm}_{Z_\sigma}^{s+1}$, for at least three linearly independent forms $\ell_\tau$, for faces $\sigma\in\Delta$ and $\tau\in\Delta_1^0$ containing $\gamma$. 
		Then, also in this case $\tJ(\gamma)\subseteq \hat\fm_{\gamma}^{s+1}$ and $\tJ(\gamma)\subseteq \hat\fm_{\gamma}^{s+1}$, and 
		the argument used in Case 1 leads to the dimension formula for $\tJ(Z_\sigma)_d$.
		
		\smallskip 
		
		\item[Case 3] Let  $B_\tau$ be the vertex on the (interior of the) edge $\tau\in\Delta_1^\circ$. 
		The ideal $\tJ(B_\tau)$ is generated by the sum of four ideals, two of the form $\langle \ell_\tau^{s+1}\rangle\cap \hat\fm_{Z_{\sigma}}^{s+1}=\langle\ell^{s+1}_\tau\rangle$, for $\tau=[B_\tau,Z_\sigma]$, and two of the form $\langle \ell_\tau^{r+1}\rangle\cap\hat\fm_{\gamma}^{s+1}$, for $\tau=[B_\tau,\gamma]$.
		By a change of coordinates we may assume that $\ell_{[B_\tau,\gamma]} = x$, $\ell_{[Z_\sigma,\gamma]} = y$ and $\ell_{[Z_\sigma,\gamma']} = z$. Then,
		\begin{equation}\label{eq:Btau}
			\tJ(B_\tau)=\langle (y+az)^{s+1},x^{s+1-i}y^i,x^{s+1-i}z^i\colon 0\leq i\leq s-r\rangle\,,
		\end{equation}
		for some $a \in \R$.
		We use the following lemma to compute the dimension of this ideal in degree $d \geq 2s-r$.
		\begin{lemma}\label{lem:dimJB}
			Let $\tJ(B_\tau)$ be as in \eqref{eq:Btau} and $d \geq 2s-r$.
			Then,
			\begin{multline}\label{eq:dimJB}
				\dim \tJ(B_\tau)_d =\\ 2(s-r+1)\binom{d-s+1}{2} - \bigr(2(s-r)+1\bigr)\binom{d-s}{2}
				- \binom{d-s-r}{2} + \binom{d-2s-r}{2}.
			\end{multline}
		\end{lemma}
		\begin{proof}
			If $d \geq 2s-r$, all monomials in $\langle x^{s+1-i}y^i:i=0,\dots,s+1 \rangle_d$ can be generated as elements in $\langle x^{s+1-i}(y+az)^i \rangle_d + \langle y^{s+1} \rangle_d + \langle z^{s+1} \rangle_d$.
			Then $\tJ(B_\tau)_d=\tJ_d$, where $\tJ\subseteq\tS$ is the ideal given by
			\begin{equation*}
				\tJ=\langle y^{s+1}, x^{s+1-i}y^i,x^{s+1-i}z^i\colon 0\leq i\leq s-r \rangle\,. 
			\end{equation*}
			In particular, $\dim \tJ(B_\tau)_d=\dim \tJ_d$ for $d \geq 2s-r$.		
			Similarly as in Lemmas \ref{lem:dim_edge_ideal} and \ref{lem:dimEborder}, we get the exact sequence
			\begin{multline}\label{eq:sequence}
				\renewcommand{\arraystretch}{1.5}
				0 
				\rightarrow \bigoplus_{i=1}^{s-r}\tS(-s-i-2)
				\xrightarrow{\phi_3} \begin{array}{l}
					\bigoplus_{i=1}^{2(s-r)}\tS(-s-2)\\
					\bigoplus_{i=1}^{s-r}\tS(-s-i-1)\\
					\bigoplus\tS(-s-r-2)
				\end{array}
				\xrightarrow{\phi_2} \bigoplus_{i=1}^{2(s-r+1)}\tS(-s-1)
				\xrightarrow{\phi_1} \tS
				\rightarrow \tS/\tJ\rightarrow 0.
			\end{multline}
			The functions in \eqref{eq:sequence} can be described as follows, 
			\begin{itemize}
				\renewcommand{\labelitemi}{\Tiny$\blacksquare$}
				\item $\phi_1 = [y^{s+1}, x^{s+1}, x^{s}y, \cdots, x^{r+1}y^{s-r}, x^{s}z, \cdots, x^{r+1}z^{s-r}]$;
				\item $\phi_2 = [\phi_{21}~|~\phi_{22}~|~\phi_{23}~|~\phi_{24}]$ is a matrix with $3(s-r)+1$ columns where
				\begin{itemize}
					\item the $i$-th column of $\phi_{21}$ corresponds to the relation between $x^{s+2-i}y^{i-1}$ and $x^{s+1-i}y^{i}$, $i = 1, \dots, s-r$,
					\item the $i$-th column of $\phi_{22}$ corresponds to the relation between $x^{s+2-i}z^{i-1}$ and $x^{s+1-i}z^{i}$, $i = 1, \dots, s-r$,
					\item the $i$-th column of $\phi_{23}$ corresponds to the relation between $x^{s+1-i}y^i$ and $x^{s+1-i}z^{i}$, $i = 1, \dots, s-r$,
					\item the only column of $\phi_{24}$ corresponds to the relation between $y^{s+1}$ and $x^{r+1}y^{s-r}$;
				\end{itemize}
				\item $\phi_3$ is a matrix with $s-r$ columns where
				\begin{itemize}
					\item the first column corresponds to the relation between the first columns of $\phi_{21}$, $\phi_{22}$ and $\phi_{23}$,
					\item the $i$-th column, $i \geq 2$, corresponds to the relation between the $i$-th columns of $\phi_{21}$, $\phi_{22}$ and $\phi_{23}$, and the $(i-1)$-th column of $\phi_{23}$.
				\end{itemize}
			\end{itemize}	
			Notice that if $s=r$, then $\tJ$ is an ideal in two variables and  the sequence in \eqref{eq:sequence} reduces to that in \cite[Theorem 2.7]{geramita1998fat} with $t=2$ (the number of edges with different slopes at the vertex $B_\tau$).
			From \eqref{eq:sequence} we get 
			\begin{multline*}
				\dim \tJ_d = 2(s-r+1)\binom{d-s+1}{2} - 2(s-r)\binom{d-s}{2}-\sum_{i=1}^{s-r}\binom{d-s-i+1}{2}\\
				\quad\quad - \binom{d-s-r}{2} + \sum_{i=1}^{s-r}\binom{d-s-i}{2}\,,
			\end{multline*}
			and this yields the dimension formula in Equation \eqref{eq:dimJB}.
		\end{proof}
	\end{description}
	\subsubsection*{Vanishing homology: } 
	We now prove that $H_0(\calJ)_d=0$ for every $d\geq2s-r+1$.
	Recall that by \cite[Lemma 3.3]{schenck1997local} (see Lemma \ref{lemma:ordering}) we can always choose a triangle in $\Delta$ with two vertices on the boundary.
	Let $\sigma\in\Delta$ be such a triangle, we denote its vertices as in Figure \ref{fig:MSsplit} (right), with the edge $[\gamma',\gamma'']$ lying on the boundary of $\Delta$.
	
	First, for each vertex contained in $\sigma$ we select a subset of interior edges in $\Delta_1^\star$ such that the vertex ideal can be generated by the sum of these edge ideals.
	Specifically, for the vertex $\gamma$ we will choose the edges $[\gamma,B_\tau]$, $[\gamma,B_{\tau'}]$ and $[\gamma,Z_\sigma]$, and for the vertex $Z_\sigma$ we will take the three edges connecting $Z_\sigma$ to the boundary.
	Denote by $\sigma'\in\Delta_2$ the triangle adjacent to $\sigma$ such that $\sigma\cap\sigma'=\tau$, and $\tau=[\gamma,\gamma']$.
	
	Up to a change of coordinates, we may assume that $\ell_{\tau}=x$ and $\tJ(B_\tau)$ is the sum of the ideals
	\begin{align*}\tJ([B_\tau, Z_\sigma])&=\tJ([B_\tau,Z_{\sigma'}])=\langle (y+az)^{s+1}\rangle,\\
		\tJ([B_\tau,\gamma])&=\langle x^{s+1-i}y^i\colon 0\leq i\leq s-r\rangle, \text{\; and}\\ 
		\tJ([B_\tau,\gamma'])&=\langle x^{s+1-i}z^i\colon 0\leq i\leq s-r\rangle.
	\end{align*}
	Let $\partial_1$ be the boundary map in the complex $\calJ$ of $\Delta^\star$ i.e, 
	\[\partial_1\colon\bigoplus_{\tau\in{(\Delta^\star)}_1^\circ}\tJ(\tau)\rightarrow\bigoplus_{\gamma\in{(\Delta^\star)}_0^\circ}\tJ(\gamma)\,.\]
	If $d\geq 2s-r+1$ and $g=f x^{s+1-i}y^iz^i$ is polynomial of degree $d$ for $f\in\tS$ and $0\leq i\leq s-r$, then we have $g\in\tJ([B_\tau,\gamma])\cap \tJ([B_\tau,\gamma'])$, and  \[\partial_1(g[B_\tau,\gamma])+\partial_1(g[B_\tau,\gamma'])=g[\gamma]\;.\] 
	A similar follows for $B_{\tau'}$ and $\gamma''$. This together to {Case 2} leads to
	\[\tJ(\gamma)_d\subseteq  \partial_1\bigl(\tJ([\gamma,Z_\sigma]),\tJ([B_\tau,\gamma]),\tJ([B_\tau,\gamma'],\tJ([B_{\tau'},\gamma]),\tJ([B_{\tau'},\gamma''])\bigr)_d\bigr|_{\gamma}.\]
	Moreover, from {Case 1},
	\[\tJ(Z_\sigma)_d\subseteq \partial_1(\tJ([Z_\sigma,\gamma'']),\tJ([Z_\sigma,\gamma']), \tJ([Z_\sigma,B_{\tau''}]))_d|_{Z_\sigma},\] 
	and from {Case 3},
	\[\tJ(B_\tau)_d\subseteq\partial_1\bigl(\tJ([Z_\sigma,B_\tau]),\tJ([B_\tau,\gamma],\tJ([B_\tau,\gamma'])_d\bigr|_{B_\tau}.\]
	Therefore, for $d\geq 2s-r+1$ the graded piece at degree $d$ of each ideal associated to an interior vertex in $\sigma$ is contained in $\im(\partial_1)_d$. 
	
	Notice that if $\Delta$ is composed of only one triangle then the only interior vertex is $Z_\sigma$ and this implies  $H_0(\calJ)_{d}=0$.
	If not, we take a triangle $\sigma'\in\Delta\setminus\{\sigma\}$ with two vertices on the boundary of $\Delta\setminus\{\sigma\}$, and apply the previous argument to the complex $\Delta^\star\setminus\{\sigma\}$.
	After $f_2$-steps (equal to the number of triangles in $\Delta$), we will have considered all the interior vertices of $\Delta^\star$.
	We conclude that $H_0(\calJ)_{d}=0$ for any simplicial complex $\Delta$ with a finite number of triangles. 
	
	Then, if $s\geq\max\{r,2r-1\}$ and $d\geq 2s-r+1$, the dimension formula in Equation \eqref{eq:PS} can explicitly be written as 
	\begin{align}
		\dim\calS_{d}^{\br,\bs}(\Delta^\star)= \binom{d+2}{2} &+ 3f_2\binom{d-s+1}{2}+3f_2\biggl[(d-s)^2- \binom{d-2s+r}{2} \biggr]\label{eq:dimgenPS}\\
		&+2f_1^\circ\biggl[ (s-r+1) \binom{d-s +1}{2} -(s-r)\binom{d-s}{2}\biggr]\nonumber\\
		&-(f_0^\circ+f_2)\biggl[ \binom{d+2}{2}-\binom{s+2}{2}\biggr]
		-f_1^\circ\dim \tJ(B_\tau)_d\,,\nonumber
	\end{align}
	where $\dim \tJ(B_\tau)_d$ is given in Equation \eqref{eq:dimJB}.
	
	In particular,  for $s = 2r-1$ and $d=3r-1$, the Euler relations $f_1^\circ=2f_2-f_0+1$ and $f_0^\circ=f_2-f_0+2$ applied to \eqref{eq:dimgenPS} lead to 
	\begin{equation}\label{eq:dimPS}
		\dim\calS_{3r-1}^{\br,\bs}(\Delta^\star)=\frac{1}{2}r(r-1)f_2+r(2r+1)f_0
		\,.
	\end{equation}
	The dimension formula \eqref{eq:dimPS} was proved by Speleers in \cite[Theorem 5]{speleers_2013}.

	\begin{table*}[ht]
		\centering
		\begin{tabular}{|>{\centering\arraybackslash}m{1cm}|>{\centering\arraybackslash}m{1.5cm}|>{\centering\arraybackslash}m{2cm}|>{\centering\arraybackslash}m{1.5cm}|>{\centering\arraybackslash}m{1.5cm}|>{\centering\arraybackslash}m{2.5cm}|}
			\hline\hline
			\rowcolor{gray!10}
			\thead[c]{$(r,s)$} & \thead[c]{$d$} & \thead[c]{$\dim H_0(\calJ)_d$} & \thead[c]{$\mathrm{LB}\eqref{eq:lowerbound_computable}$} & \thead[c]{$\mathrm{LB}\eqref{eq:lowerbound}$} &
			\thead[c]{$\dim \calS_d^{r,s}(\Delta^\star)$} \\
			\hline\hline
			& 4 & 9 & 15 & 15 & 16 \\
			\cline{2-6}
			& 5 & 0 & 67 & 67 & 67 \\
			\cline{2-6}
			\multirow{-3}{*}{$(2,3)$} & 6 & 0 & 160 & 160 & 160 \\
			\hline
			\hline
			& 5 & 16 & 21 & 21 & 22 \\
			\cline{2-6}
			& 6 & 0 & 54 & 54 & 54 \\
			\cline{2-6}
			\multirow{-3}{*}{$(3,4)$} & 7 & 0 & 138 & 138 & 138 \\
			\hline
			\hline
			& 7 & 1 & 42 & 42 & 43 \\
			\cline{2-6}
			& 8 & 0 & 147 & 147 & 147 \\
			\cline{2-6}
			\multirow{-3}{*}{$(3,5)$} & 9 & 0 & 285 & 285 & 285 \\
			\hline
			\hline
		\end{tabular}
		\caption{
			The triangulation $\Delta$ is the Powell--Sabin 6-split shown in Figure \ref{fig:MSsplit}.
			The lower bounds $\mathrm{LB}\eqref{eq:lowerbound}$ and $\mathrm{LB}\eqref{eq:lowerbound_computable}$ coincide for the shown choices of $(r, s, d)$, and they both coincide with $\dim \calS_d^{r,s}(\Delta^\star)$ for large enough degree.
			The dimension of $\calS_d^{r,s}(\Delta^\star)$ was computed in \cite{speleers_2013} for $(r,s,d) \in \{ (2,3,5), (3,5,8)\}$, and that for $(r,s) = (3,4)$, we compute $\dim \calS_d^{r,s}(\Delta^\star)$ using Macaulay2 \cite{M2}.
		}
	\end{table*}
	\section{Concluding Remarks}\label{sec:conclusion}
	We have demonstrated how methods from homological algebra can be used to compute the dimension of supersmooth spline spaces on general triangulations; in particular, we have proved a combinatorial formula for the dimension of superspline spaces in sufficiently large degree.
	We also illustrated how homological algebra methods can be used to reproduce a variety of results from the literature \cite{chui85,sorokina_2013,floater_hu_2020,speleers_2013}, as well as generalizing some of them \cite{speleers_2013}. This opens several directions for future research:
	\begin{itemize}[style=unboxed,leftmargin=0cm]\renewcommand{\labelitemi}{\scalebox{0.5}{$\blacksquare$}}
		\item Supersmoothness can help defining spline spaces with both stable dimension and locally supported basis functions, retaining full approximation power and avoiding prohibitively high degrees.
		Consequently, in the future these methods should be combined with constructive approaches to build spline spaces that are useful for the finite element method, such as triangulations and T-meshes.
		
		\item As it was noted by Schenck in \cite{schenck2016algebraic}, the algebraic tools developed for the study of spline spaces on polyhedral complexes with uniform global smoothness and mixed supersmoothness across the codimension-1 faces had not been extended to the case we study in this paper. 
		As we observed, the algebraic approach to the dimension problem of splines with mixed supersmoothness at higher codimension faces of the partition leads to the consideration of ideals generated by products of powers of linear forms in several variables.
		In the case of generic forms, this type of ideals has been recently studied by DiPasquale, Flores, and Peterson in  \cite{dipasquale_flores_Peterson_2020} via apolarity. 
		It will be interesting to extend this approach to ideals generated by arbitrary products of powers of linear forms to study full vertex ideals and derive an improved lower bound, as well as deriving an upper bound	on the dimension of superspline spaces. While we have provided simple and computable lower bounds on the dimension, they only consider a simplified version of the vertex ideals at play.
		Considering the full vertex ideals is a first research direction that should be explored.
		\item The lower bound on $\dim \calS_d^{\br,\bs}(\Delta)$ proved in Theorem \ref{thm:lowerbound_alternate} gives the exact dimension of the superspline space in large enough degree $d$. It would be interesting to find the smallest value of $d$ from which the dimension formula holds; for this, results by Ibrahim and Schumaker in \cite{ibrahim91} might give a good estimate on the smallest degree for which homology term $H_0(\calJ)_d$ vanishes. 
		The analysis of the quotient of the vertex ideals $\overline\tJ(\gamma)/\tJ(\gamma)$ relates to the study of  intrinsic smoothness properties of splines. An estimate on the smallest degree for which $\dim \tJ(\gamma)_d=\dim \overline\tJ(\gamma)_d$ will also contribute to a better understanding of $\dim \calS_d^{\br,\bs}(\Delta)$, and it would be interesting to explore the implications of this algebraic approach combined with the results and techniques developed in \cite{sorokina_2010,sorokina_2013,floater_hu_2020} for intrinsic supersmoothness using Bernstein-B\'ezier methods.
	\end{itemize}
	\subsection*{Acknowledgements} We would like to thank Michael DiPasquale for providing many helpful comments and suggestions.  	
	\appendix
	
	\section{}\label{appendix}
	This appendix is devoted to the detailed proof of Lemma \ref{lemma:homology}. 	The result follows by a slight modification of the proof by Schenck and Stillman in \cite[Lemma 3.2]{schenck1997local} in the case of splines $\calS_d^r(\Delta)$ with global uniform smoothness $r$. 
	
	First, we recall the following lemma.   
	\begin{lemma}[{\cite[Lemma 3.3]{schenck1997local}}]\label{lemma:ordering} If $\Delta\subseteq\R^2$ is a triangulation, then there exists a total order $\succ$ on $\Delta_0$ such that for every $\gamma\in\Delta_0^\circ$ there exist vertices $\gamma'$ and $\gamma''$ adjacent to $\gamma$, with $\gamma\succ \gamma', \gamma''$ and such that the edges $\tau=[\gamma,\gamma']$ and  $\tau=[\gamma,\gamma'']$ have different slopes.
	\end{lemma}
	\begin{proof}[Proof of Lemma~\ref{lemma:homology}] 
		We show the claim by proving that $H_0(\calJ)$ has finite length.
		For a vertex $\gamma\in\Delta_0^\circ$ and $f\in\tJ(\gamma)$, we denote by $f[\gamma]$ the corresponding element in $H_0(\calJ)$, where $\calJ$ is the complex of ideals defined in \eqref{eq:complexJ}. 
		If $\gamma$ is a boundary vertex we write $f [\gamma] = 0$ for any $f\in\tS$.
		We prove that $H_0(\calJ)$ has finite length by showing that there exist a sufficiently large integer $N$ such that $\langle x,y,z\rangle^Nf[\gamma]=0$ in $H_0(\calJ)$ for all $\gamma\in\Delta_0^\circ$ and  $f\in\tJ(\gamma)$.  
		
		First, we fix a total ordering ``$<$" on the  vertices of $\Delta$ such that for each vertex $\gamma\in\Delta_0^\circ$ there are two edges $\tau'=[\gamma,\gamma']$ and $\tau''=[\gamma,\gamma'']$ with different slopes such that each of the vertices $\gamma'$ and  $\gamma''$ is either on the boundary of $\Delta$ or is $<$ than $\gamma$. We know that such an ordering exists by Lemma \ref{lemma:ordering}.
		Take such a vertex $\gamma\in\Delta_0^\circ$, and suppose $\langle x,y,z\rangle^Ng[\nu]=0$ for all vertices $\nu<\gamma$ and $g \in \tJ(\nu)$, for some integer $N>0$.
		
		If $\tau'=[\gamma,\gamma']$, let $s=\max\{s_{\gamma'}, s_{\gamma},r_{\tau'}\}$. By definition \eqref{eq:edgesgen}, the ideal $\tJ(\tau')= \langle \ell_\tau^{r+1}\rangle\cap \hat\fm_\gamma^{s_\gamma+1}\cap \hat\fm_{\gamma'}^{s_{\gamma'}+1}$, where $\hat\fm_{\gamma}$ and $\hat\fm_{\gamma'}$ are the ideals of polynomials in $\tS$ vanishing at $\hat{\gamma}$ and $\hat{\gamma}'$, respectively. 
		In particular, $\ell_{\tau'}^{s+1}\in\tJ(\tau)$ and thus
		\begin{equation}\label{eq:prooflemma}
			\ell_{\tau'}^{s+1}f[\gamma]
			=
			\ell_{\tau'}^{s+1}f[\gamma']\,.
		\end{equation}
		Since $\ell_{\tau'}$  
		is a linear form in $\tS$ and $\gamma'\in\partial\Delta$ or $\gamma'<\gamma$, by \eqref{eq:prooflemma} it follows that some power of $\ell_{\tau'}$ annihilates $f[\gamma]$. 
		Similarly, some power of $\ell_{\tau''}$
		annihilates $f[\gamma]$.
		
		On the other hand, if $f\in\tJ(\gamma)$ is given by $f=\sum_{\tau\ni\gamma} f_\tau$ for $f_\tau\in\tJ(\tau)$, we have
		\[f[\gamma]=\sum_{\tau=[\gamma,\gamma(\tau)]}f_\tau[\gamma(\tau)]\,,\]
		where $\gamma(\tau)\in\tau$ denotes the vertex adjacent to $\gamma$ on the edge $\tau$.
		
		For an edge $\tau\in\lk(\gamma)$, denote by $\ell_{\tau}$ a choice of a linear form vanishing on $\hat{\tau}$. Define $p=\prod\bigl\{\ell_\tau\colon\tau\in\lk(\gamma)\bigr\}$ and $s'=\max\bigl\{s_\nu,r_\tau\colon \nu,\tau\in\St(\gamma)\bigr\}$. Then, by construction $p(\gamma)\neq 0$ and $p^{s'+1}f[\gamma]=p^{s'+1} f[\nu]$ in $H_0(\calJ)$ for any vertex $\nu\in\lk(\gamma)$.
		In particular, if $\nu=\gamma(\tau')=\gamma'$ we have $p^Nf[\gamma']=0$, and therefore $p^Nf[\gamma]=0$.
		
		Hence, some power of $\ell_{\tau'}$,  $\ell_{\tau''}$, and $p$ annihilate $f[\gamma]$. But $\langle x,y,z \rangle^{M}\subseteq \langle \ell_{\tau'}, \ell_{\tau''}, p \rangle$ for a sufficiently large integer $M$, and it follows $\langle x,y,z\rangle^Mf[\gamma]=0$. 
	\end{proof}	

\end{document}